\documentclass[a4paper,draft,reqno,12pt]{amsart}
\usepackage[english]{babel}
\usepackage{amsmath}
\usepackage{amssymb}
\usepackage{amscd}
\usepackage{amsthm}
\usepackage{euscript}
\newtheorem{theor}{Theorem}[section]
\newtheorem{prop}[theor]{Proposition}
\newtheorem{lem}[theor]{Lemma}

\newtheorem{cor}[theor]{Corollary}
\theoremstyle{definition}
\newtheorem{de}[theor]{Definition}
\newtheorem{ex}[theor]{Example}
\theoremstyle{remark}
\newtheorem {re}[theor]{Remark}

\DeclareMathOperator{\Aut}{Aut}

\def\Ker{{\rm Ker}\,}

\def\HH{{\mathbb H}}
\def\SAut{{\mathrm{SAut}}}
\def\UU{{\mathbb U}}
\def\GG{{\mathbb G}}

\def\CC{{\mathbb C}}
\def\KK{{\mathbb K}}
\def\LL{{\mathbb L}}
\def\TT{{\mathbb T}}
\def\ZZ{{\mathbb Z}}

\def\DD{\mathbb{D}}
\def\VV{\mathbb{V}}

\begin{document}
\date{}
\title[Automorphisms of Danielevski varieties]{Automorphisms of Danielewski varieties}
\author{Sergey A. Gaifullin}
\address{Lomonosov Moscow State University, Faculty of Mechanics and Mathematics, Department of Higher Algebra, Leninskie Gory 1, Moscow, 119991 Russia; \linebreak and \linebreak
National Research University Higher School of Economics, Faculty of Computer Science, Kochnovskiy Proezd 3, Moscow, 125319 Russia}
\email{sgayf@yandex.ru}



\maketitle

\begin{abstract}
In 2007, Dubouloz introduced Danielewski varieties. Such varieties generalize Danielewski surfaces and provide counterexamples to generalized Zariski cancellation problem in arbitrary dimension. In the present work we describe the automorphism group of a Danielewski variety. This result is a generalization of a description of automorphisms of Danielewski surfaces due to Makar-Limanov.
\end{abstract}

\section{Introduction}

We assume that the base field $\KK$ is an algebraically closed field of characteristic zero. Let $V$ and $W$ be affine algebraic varieties over $\KK$. Generalized Zariski cancellation problem asks whether $V\times \KK\cong W\times \KK$ implies $V\cong W$. In 1989 Danielewski \cite{Da} introduced a class of surfaces given by $xy^n=P(z)$. Nowadays such varieties  are called Danielewski surfaces. Danielewski surfaces  corresponding to $n=1$ and $n=2$  provide a counterexample to the generalized Zariski cancellation problem. One of the possible ways to prove that varieties $\VV\left(xy-P(z)\right)$ and $\VV\left(xy^2-P(z)\right)$ are nonisomorphic is suggested by Makar-Limanov in \cite{ML1,ML}. He computes the intersection of all kernels of locally nilpotent derivations in both cases and shows that, if $n=1$, the intersection is the field $\KK$, and if $n=2$, the intersection is $\KK[y]$. This intersection later was named {\it Makar-Limanov invariant}. 

Moreover,  Makar-Limanov in \cite{ML1,ML} computes generators of automorphism groups of Danielewski surfaces. This inspired a lot of works on automorphisms of Danielewski surfaces and their generalizations. In~\cite{C} Crachiola investigated automorphism groups of surfaces of the form $\VV(xy^n-z^2-h(y)z)$ over a field of arbitrary characteristic. In~\cite{DP} Dubouloz and Poloni considered a more general class of surfaces $\VV(xy^n-Q(y, z ))$, where $n\geq 2$ and $Q(y, z )$ is a polynomial with coefficients in an arbitrary base field such that $Q(0,z)$ splits with $r \geq 2$ simple roots. Generators of their automorphism groups were computed. In~\cite{Po} classifications of varieties of the form $\VV(xy^n-Q(y, z ))$ up to isomorphism and up to automorphism of ambient affine 3-space are obtained. In \cite{BV} Bianchi and Veloso considered one more generalization, surfaces $\VV(xf(y)-Q(y,z))$, where $\deg f\geq 2$ and 
$$Q(y,z)=z^d+s_{d-1}(y)z^{d-1}+\ldots+s_0(y),\qquad d\geq 2.$$ For such surfaces it was proved that the Makar-Limanov invariant is $\KK[y]$. Using this generators of the automorphism group of the varieties $\VV(xf(y)-Q(z))$ were computed.

In 2007, Dubouloz \cite{D} considered another generalization of Danielewski surfaces. He introduced  varieties given by equations of the form 
$$
xy_1^{k_1}\ldots y_m^{k_m}=z^d+s_{d-1}(y_1,\ldots, y_m)z^{d-1}+\ldots+s_0(y_1,\ldots y_m),
$$
which was called {\it Danielewski varieties}.
He proved that if $d\geq 2$ and all $k_i\geq 2$, then the Makar-Limanov invariant of such a variety is equal to $\KK[y_1,\ldots, y_m]$. Using this Dubouloz found counterexamples to generalized Zariski cancellation problem in arbitrary dimension. 

In this paper we describe automorphism groups of Danielewski varieties.  To do this we need to investigate automorphism groups of other classes of varieties. So, we proceed in three steps.

Sections 2-4 contain preliminaries and lemmas, which we use later. 

First of all, in Section~\ref{no}, we consider varieties of the form $Z=\VV(y_1^{k_1}\ldots y_m^{k_m}-P(z))$.  There exists an $(m-1)$-dimensional algebraic torus action on such a variety. We elaborate some technique, which shows that if all $k_i\geq 2$, then $Z$ is {\it rigid}, that is it does not admit any nontrivial locally nilpotent derivations. Then we compute the generators of the automorphism group of $Z$, see Theorem~\ref{vtor}. We prove that the automorphism group $\Aut(Z)$ is a semidirect product of a diagonalizable group multiplying variables by constants and a finite group permuting variables. Therefore, $\Aut(Z)$ is a linear algebraic group. Moreover, it is a finite extension of an algebraic torus. For generic variety $Z$, the automorphism group is a torus, but in special cases the finite part can be rather large. We obtain criteria of commutativity, connectivity and solvability of $\Aut(Z)$.

Secondly, in Section~\ref{oo}, we investigate varieties of the form $Y=\VV(xy_1^{k_1}\ldots y_m^{k_m}-P(z))$ with  all $k_i\geq 2$. Such varieties are particular cases of Danielewski varieties. In some sense these particular cases  are the most interesting ones because they provide counterexamples to generalized Zariski cancellation problem.
Using rigidity of $Z$ we prove that the Makar-Limanov invariant of $Y$ equals $\KK[y_1,\ldots, y_m]$ and all locally nilpotent derivations are replicas of a {\it canonical} one, see Definitin \ref{cd}. We compute generators of $\Aut(Y)$, see Proposition~\ref{drep}. The automorphism group of $Y$ is generated by exponents of replicas of the canonical locally nilpotent derivation, the quasitorus acting by multiplying variables by constants, and swappings of variables with coinciding $k_i$.   Moreover, $\Aut(Y)$
is isomorphic to a semidirect products of these subgroups, see Theorem~\ref{vtt}. We prove that $\Aut(Y)$ is never commutative and it is solvable if and only if there are no five variables $y_i$ with coinciding~$k_i$.

In Section~\ref{las}, we use technique inspired by \cite{ML} to compute the group of automorphisms of an arbitrary Danielewski variety $X$.
The main idea is to consider a filtration on $\KK[X]$ and to prove that the associated graduate algebra $\mathrm{Gr}(\KK[X])$ is isomorphic to $\KK[Y]$ for some~$Y$ considered in  Section~\ref{oo}. Then we prove that every automorphism of $X$ respects this filtration, see Lemma~\ref{filtr}. Using this we introduce a homomorphism $\Phi\colon \Aut(X)\rightarrow\Aut(Y)$. This allowes to describe automorphisms of $X$ using the description of automorphisms of~$Y$, see Theorem~\ref{main}. Again all locally nilpotent derivations on $X$ are replicas of a {\it canonical} one, see Definition~\ref{cade}. The automorphism group of $X$ is isomorphic to a semidirect product of commutative group consisting  of exponents of all replicas of the canonical locally nilpotent derivation and a  {\it canonical group} $\mathbb{G}$ of $X$. The group $\mathbb{G}$ is a finite extension of a torus. Every element of $\mathbb{G}$ permutes the variables and multiplies them by constants. For generic variety~$X$ the group $\mathbb{G}$ is trivial and $\Aut(X)$ is commutative. We prove that it is a criterium of commutativity of $\Aut(X)$. Also we give a sufficient condition of its solvability.

The author is grateful to Ivan Arzhantsev for useful discussions.

\section{Derivations}

Let $A$ be a commutative associative algebra over $\KK$.
A linear mapping $\partial\colon A\rightarrow A$ is called a {\it derivation} if it satisfies the Leibniz rule $\partial(ab)=a\partial(b)+b\partial(a)$ for all $a,b\in A$.
A derivation is {\it locally nilpotent} (LND) if for every $a\in A$ there is a positive integer $n$ such that $\partial^n(a)=0$.
A derivation is {\it semisimple} if there exists a basis of $A$ consisting of $\partial$-semi-invariants. Recall that $a\in A$ is a $\partial$-semi-invariant if $\partial(a)=\lambda a$ for some $\lambda\in\KK$.

If we have an algebraic action of the additive group $(\KK,+)$ on $A$, we obtain a one-parameter subgroup in $\Aut(A)$. Then its tangent vector at unity is an LND. Exponential mapping defines a bijection between LNDs and elements in subgroups of $\mathrm{Aut}(A)$ isomorphic to~$(\KK,+)$. Similarly, if  we have an action of the multiplicative group~$(\KK^\times,\cdot)$ on~$A$, then its tangent vector at unity is a semisimple derivation.

Let $F$ be an abelian group.
An algebra $A$ is called {\it $F$-graded} if 
$$A=\bigoplus_{f\in F}A_f,$$ 
and $A_fA_g\subset A_{f+g}$ for all $f$ and $g$ in $F$.
A derivation $\partial\colon A\rightarrow A$ is {\it $F$-homogeneous of degree $f_0\in F$}, if for every $a\in A_f$ we have $\partial(a)\in A_{f+f_0}$.

A derivation $\partial\colon A\rightarrow A$ is called {\it locally bounded}, if any element $a\in A$ is contained in a $\partial$-invariant linear subspace $V\subset A$ of finite dimension.
It is easy to see that all semisimple derivations and all LNDs are locally bounded.

Let A be a finitely generated $\ZZ$-graded algebra and $\partial$ be a derivation of $A$. Then $\partial$ can be decomposed  in to  $\partial=\sum_{i=l}^k\partial_i$, where $\partial_i$ is a homogeneous derivation of degree $i$. Indeed, let $a_1,\ldots,a_m$ be generators of $A$. Let $\deg a_j=u_j$. Then $$\partial (a_j)=\sum_{i=l_j}^{k_j}b_{i+u_j},$$ where $b_{i+u_j}\in A_{i+u_j}$. Denote $l=\mathrm{min}\{l_1,\ldots,l_m\}$, $k=\mathrm{max}\{k_1,\ldots,k_m\}$. Using the Leibniz rule we obtain 
$$\forall a\in A_j \qquad\partial (a)\in\bigoplus_{i=l}^k A_{j+i}.$$ So we have $\partial=\sum_{i=l}^k\partial_i$, where $\partial_i\colon A_j\rightarrow A_{j+i}$ is a linear mapping. The Leibniz rule for $\partial_j$ follows from the Leibniz rule for $\partial$.

\begin{re}
Further when we write $\partial=\sum_{i=l}^k\partial_i$, we assume that $\partial_l\neq 0$ and $\partial_k\neq 0$.
\end{re}

\begin{lem}\label{fl}(See \cite{Re} for (1) and \cite[Lemma~3.1] {FZ} for (2))
Let $A$ be a finitely generated $\ZZ$-graded algebra. Assume $\partial\colon A\rightarrow A$ is a derivation. We have $\partial=\sum_{i=l}^k\partial_i$, where $\partial_i$ is a homogeneous derivation of degree $i$. Then

1) if $\partial$ is an LND then $\partial_l$ and $\partial_k$ are LNDs.

2) if $\partial$ is locally bounded then if $l\neq 0$, $\partial_l$ is an LND, and if $k\neq 0$, $\partial_k$ is an LND.
\end{lem}

Lemma \ref{fl}(1) implies the following lemma.

\begin{lem}\label{flz}
If a $\ZZ$-graded algebra $A$ admits a nonzero LND, then $A$ admits a nonzero $\ZZ$-homogeneous LND. 
\end{lem}
Let $F\cong\ZZ^n$. Assume $A$ is a finitely generated $F$-graded algebra. Applying the result of Lemma~\ref{flz} $n$ times we obtain the following corollary.
\begin{lem}\label{flza}
 The algebra $A$ admits a nonzero LND if and only if $A$ admits a nonzero $F$-homogeneous LND. 
\end{lem}

Let $X$ be an affine algebraic variety and $A=\KK[X]$ be the algebra of regular functions on~$X$. Then $\ZZ^n$-gradings on $\KK[X]$ are in bijection with actions of $n$-dimensional algebraic torus $T=(\KK^\times)^n$ on $X$. 
If we have a $\ZZ^n$-homogeneous derivation of $\KK[X]$, we call it {\it $T$-homogeneous derivation}.

Let $F$ and $S$ be abelian groups. Consider an $F$-graded algebra
$$A=\bigoplus_{f\in F}A_f.$$ 
Let $\pi\colon F\rightarrow S$ be a homomorphism. 
Then $A$ can be considered as $S$-graded algebra via
$$A=\bigoplus_{s\in S}A_s, \qquad\text{where}\ A_s=\bigoplus_{\pi(f)=s}A_f.$$ 

\begin{lem}  \label{grub}
If $\partial$ is an $F$-homogeneous derivation of degree $f_0$, then~$\partial$ is an $S$-homogeneous derivation of degree $s_0=\pi(f_0)$.
\end{lem}
\begin{proof}
Let $a\in A_s$. Then 
$$a=\sum_{\pi(f)=s}a_f,$$ where $a_f\in A_f$. We have $\partial(a)=\sum(\partial(a_f))$. Since $\partial$ is $F$-homogeneous of degree $f_0$, we obtain $\partial(a_f)\in A_{f+f_0}$. But 
$$\pi(f+f_0)=\pi(f)+\pi(f_0)=s+s_0.$$ So, $\partial(a)\in A_{s+s_0}$. 
\end{proof}

Let us use the notation $a\mid b$ if $a$ divides $b$.  

\begin{lem}\label{semisimp}
Assume $A$ is a domain. Let $\delta$ be a semisimple derivation, corresponding to a subgroup $\Lambda~\cong~\KK^\times$ of $\Aut(A)$. Suppose for some $f\in A$ we have $f\mid\delta(f)$. Then $f$ is $\Lambda$-homogeneous, and hence, for every $\varphi\in\Lambda$, there exists $\lambda\in\KK\setminus\{0\}$ such that  $\varphi^*(f)=\lambda f$.
\end{lem}
\begin{proof}
We have $\delta(f)=fg$. Let us consider $\ZZ$-grading corresponding to $\Lambda$. Let $f~=~\sum_{i=l}^k f_i$ and $g=\sum_{j=p}^q g_j$ be decompositions onto homogeneous components. Then $\delta(f)~=~\sum_{u=l+p}^{k+q}\delta(f)_u$. Since $A$ is a domain, $\delta(f)_{l+p}\neq 0$ and $\delta(f)_{k+q}\neq 0$. But $\delta$ acts on $A_i$ by multiplying by $i$, that is $\delta(f)~=~\sum_{i=l}^k if_i$. So, we have $p=q=0$. Hence, $g\in A_0$. Then $\delta(f)_i=g f_i=i f_i$. Hence, $f$ is homogeneous of some degree $i$. Thus, $\varphi^*(f)=t^i f$ for some $t\in \KK\setminus\{0\}$. 
\end{proof}

We would like to investigate the automorphism group $\Aut(X)$ of an affine algebraic variety $X$. The first step is to investigate the subgroup of $\Aut(X)$ generated by all algebraic subgroups isomorphic to $(\KK,+)$. This subgroup is called the subgroup {\it of special automorphisms}. We denote it by $\mathrm{SAut}(X)$. The subring $\mathrm{ML}(X)\subset \KK[X]$ consisting of all $\mathrm{SAut}(X)$-invariants is called the  {\it Makar-Limanov invariant} of~$X$. In other words $\mathrm{ML}(X)$ is the intersection of kernels of all LNDs of $\KK[X]$. If $\mathrm{ML}(X)=\KK[X]$, that is $\KK[X]$ does not admit any nontrivial LND, then 
the variety $X$ is called {\it rigid}. 

If $\partial$ is an LND on $X$, 
then its {\it replicas} $f\partial$, 
where 
$f\in\mathrm{Ker} \partial$, are LNDs. Exponents of all replicas of an LND $\partial$ give a commutative subgroup $\UU(\partial)=\left\{\exp(f\partial)\mid f\in\Ker\partial\right\}$ in $\SAut(X)$. Let us call $\UU(\partial)$ the {\it big unipotent subgroup corresponding to} $\partial$.
An affine algebraic variety $X$ is called {\it almost rigid}, if there is an LND $\overline{\partial}\colon\KK[X]\rightarrow\KK[X]$ such that every LND on $\KK[X]$ is a replica of 
$\overline{\partial}$. If $X$ is almost rigid, then $\SAut(X)=\UU(\overline{\partial})$ is commutative.

\section{$m$-suspensions}\label{tre}

Let $X$ be an affine algebraic variety. Given a nonconstant regular function $f\in\KK[X]$, we can define a new affine variety 
$$Y=\mathrm{Susp}(X,f)=\VV\left(uv-f(x)\right)\subset \KK^2\times X$$
called a {\it suspension} over $X$.
LNDs on suspensions are investigated in \cite{AKZ}. Recall that a variety is called {\it flexible} if the tangent space at every regular point is generated by tangent vectors to orbits of some $(\KK,+)$-actions. In \cite{AFKKZ} it is proved that for an irreducible affine variety of dimension $\geq 2$ this condition is equivalent to transitivity and infinitely transitivity of $\mathrm{SAut}(X)$-action on $X$. A suspension over a flexible affine variety is flexible, see \cite[Theorem~0.2(3)]{AKZ}. Let us consider a generalization of this construction.

\begin{de}
Let $X$ be an affine variety. Given a nonconstant regular function $f\in\KK[X]$ and positive integers $k_1,\ldots,k_m$, we define a new affine variety 
$$Y=\mathrm{Susp}(X,f,k_1,\ldots, k_m)=\VV\left(y_1^{k_1}y_2^{k_2}\ldots y_m^{k_m}-f(x)\right)\subset \KK^m\times X$$
called an {\it $m$-suspension} with weights $k_1, \ldots, k_m$ over $X$.
\end{de}

Recall that a {\it quasitorus} is a direct product of a torus and a finite commutative group. We have a natural diagonal action of $m$-dimensional algebraic torus $(\KK^\times)^m$ on $\KK[y_1,\ldots,y_m]$. The stabilizer~$\HH$ of the monomial $y_1^{k_1}y_2^{k_2}\ldots y_m^{k_m}$  is isomorphic to the direct product of the $(m-1)$-dimensional torus~$\TT$ and a finite cyclic group $\ZZ_{\gcd(k_1,\ldots,k_m)}$. There is an effective action of $\HH$ on $Y=\mathrm{Susp}(X,f,k_1,\ldots, k_m)$. 
\begin{de} 
We call $\HH$ {\it the proper quasitorus} of $Y$.  
\end{de}

There is a natural action of the symmetric group $\mathrm{S}_m$ on $\KK[y_1,\ldots,y_m]$. The stabilizer~$S(Y)$ of the monomial $y_1^{k_1}y_2^{k_2}\ldots y_m^{k_m}$  is isomorphic to 
$$\mathrm{S}_{m_1}\times\ldots\times\mathrm{S}_{m_n},$$
where $m=m_1+\ldots+m_n$. 
Each $\mathrm{S}_{m_i}$ permutes $y_j$ with fixed $k_j$. There is an effective action of $S(Y)$ on $Y=\mathrm{Susp}(X,f,k_1,\ldots, k_m)$. 
\begin{de}
We call $S(Y)$ the {\it symmetric group} of $Y$. 
\end{de}

During this section we assume, that $Y=\mathrm{Susp}(X,f,k_1,\ldots, k_m)$ is irreducible. Lemma~\ref{ir} gives a criterium of $Y$ to be irreducible in case $X\cong\KK$.

$\TT$-action on $Y$ corresponds to a $\ZZ^{m-1}$-grading of $\KK[Y]$. Often it is convenient to consider $\ZZ$ -gradings, that are its coarsenings. Let us denote by $\mathfrak{h}_{ij}$ the $\ZZ$-grading given by $\deg(y_i)~=~-~k_j$, $\deg(y_j)~=~k_i$. Degrees of all other $y_n$ and degrees of all $g\in\KK[X]$ are equal to zero. 

\begin{lem}\label{edin}
Let $\partial$ be a $\TT$-homogeneous LND of $\KK[Y]$. Then there is an index $i~\in~\{1,2,\ldots,m\}$ such that for all $j\neq i $ we have $ \partial(y_j)=0$.
\end{lem}
\begin{proof}
If for all $j$ we have $\partial(y_j)=0$, the lemma is proven. Assume that $\partial(y_i)\neq 0$. 
Fix any $j\neq i$. Let us consider the $\ZZ$-grading $\mathfrak{h}_{ij}$ on $\KK[Y]$. By Lemma~\ref{grub}, $\partial$ is a $\ZZ$-homogeneous LND. If $\deg(\partial)\leq 0$, then $\deg(\partial(y_i))<0$, so $y_i\mid \partial(y_i)$. Therefore, $\partial(y_i)=0$, see for example~\cite[Proposition~1.9]{Fr}. This is a contradiction. Hence, $\deg(\partial)>0$. Then $\deg(\partial(y_j))>0$. We obtain $y_j\mid\partial(y_j)$, hence $\partial(y_j)=0$.
\end{proof}

Consider the field $\LL_j=\overline{\KK(y_j)}$, which is the algebraic closure of $\KK(y_j)$. If there is a fields embedding  $\KK\subset\LL$ and $Z$ is an affine algebraic variety over $\KK$, we denote by
$Z(\LL)$ the affine algebraic variety over $\LL$ given by the same equations as $Z$. Then we have 
$$Y(\LL_j)=\VV\left(y_1^{k_1}\ldots y_m^{k_m}-f\right)\subset \LL_j^{m-1}\times X(\LL_j).$$ 
Since $y_j$ is a constant now, $Y(\LL_j)$ is a $(m-1)$-suspension over $X(\LL_j)$.

Denote by $Y_j$ the affine algebraic variety over $\LL_j$ of the form 
$$\VV\left(y_1^{k_1}\ldots y_{j-1}^{k_{j-1}}y_{j+1}^{k_{j+1}}\ldots y_m^{k_m}-f\right)\subset \LL_j^{m-1}\times X(\LL_j).$$

\begin{re}
The variety $Y_j$ can be reducible.
\end{re}

\begin{lem}\label{pat}
If there is a $\TT$-homogeneous LND $\partial\colon\KK[Y]\rightarrow\KK[Y]$ such that $\partial(y_i)\neq 0$, then for each $j\neq i$ there exists LND $\partial_j\colon\LL_j[Y_j]\rightarrow\LL_j[Y_j]$ such that $\partial_j(y_i)\neq 0$. 
\end{lem}
\begin{proof}
By Lemma \ref{edin}, $\partial(y_j)=0$. So, $\partial$ gives a derivation of $Y(\LL_j)$. Consider the mapping $\psi\colon y_1\mapsto \varepsilon y_1$, where $\varepsilon\in\LL_j, \varepsilon^{k_1}=y_j^{k_j}$. Then $\psi$ is an isomorphism between $Y(\LL_j)$ and $Y_j$. So we obtain $\partial_j\colon\LL_j[Y_j]\rightarrow\LL_j[Y_j]$ such that $\partial_j(y_i)\neq 0$. 
\end{proof}

\begin{lem}\label{osn}
If $Y_i$ is a rigid variety, then for every nontrivial $\TT$-homogeneous LND $\partial\colon\KK[Y]\rightarrow\KK[Y]$ we have 

1) $\partial(y_i)\neq 0$, 

2) $\forall j\neq i, \partial(y_j)=0$, 

3) $\deg(\partial)>0$ with respect to the $\ZZ$-grading $\mathfrak{h}_{ij}$.
\end{lem}
\begin{proof}
If $\partial(y_i)=0$, then $\partial$ gives a nontrivial derivation of $Y(\LL_i)$. Hence, $Y_i$ is not rigid. So, we have $\partial(y_i)\neq 0$. By Lemma \ref{edin} for every~$j\neq i$ we obtain $\partial(y_j)=0$.

By Lemma \ref{grub} $\partial$ is $\mathfrak{h}_{ij}$-homogeneous. If $\deg(\partial)\leq 0$, then $\deg(\partial(y_i))<0$. Hence, $\partial(y_i)$ is divisible by $y_i$. Since $\partial$ is an LND, we obtain $\partial(y_i)=0$. A contradiction. So, we have $\deg(\partial)>0$.
\end{proof}

\begin{lem}\label{ml}
Let $Y_i$ be a rigid variety. Then for each $ j\neq i$, we have $y_j\in\mathrm{ML}(Y)$.
\end{lem}
\begin{proof}
Let $\partial\colon\KK[Y]\rightarrow\KK[Y]$ be an LND. Consider the $\ZZ$-grading $\mathfrak{h}_{ij}$ on~$\KK[Y]$. We have $\partial=\sum_{p=l}^k\partial_p$. By Lemma~\ref{fl} we have, that $\partial_l$ is LND. We can decompose $\partial_l$ on to the sum of $\TT$-homogeneous derivations. Among them there is a $\TT$-homogeneous LND with $\mathfrak{h}_{ij}$-degree $l$.
By Lemma~\ref{osn} (3) we have $l>0$. So, we have $\partial(y_j)=\sum_{p=l}^k\partial_p(y_j)$, where $\deg(\partial_p(y_j))=k_i+p\geq  k_i+l>0$. Therefore, $y_j\mid \partial(y_j)$. Hence,  $\partial(y_j)=0$. 
\end{proof}
\begin{lem}\label{aut}
Let $Y_i$ be a rigid variety. Let $\Lambda\subset\mathrm{Aut}(Y)$ be a subgroup isomorphic to the multiplicative group $(\KK^\times,\cdot)$. Let $\varphi$ be an element of~$\Lambda$.  Then for each $ j\neq i$ there exists a nonzero element $\lambda\in\KK\setminus\{0\}$ such that $\varphi^*(y_j)=\lambda y_j$.
\end{lem}
\begin{proof}
Let $\delta$ be the semisimple derivation corresponding to $\Lambda$. Then with respect to $\mathfrak{h}_{ij}$, we have $\delta=\sum_{p=l}^k\delta_p$. By Lemma \ref{fl} if $l<0$, then $\delta_l$ is an LND. This contradicts to Lemma~\ref{osn}. Therefore, $l\geq 0$. Hence, for all $p=l,\ldots,k$ we have $\deg(\delta_p(y_j))>0$. Therefore, $y_j\mid \delta (y_j)$. By Lemma~\ref{semisimp} $\varphi^*(y_j)=\lambda y_j$ for some $\lambda\in\KK\setminus\{0\}$.
\end{proof}

\section{$m$-suspensions  over a  line}

In this section we investigate $m$-suspensions  over a line. The results of this section we use in the next two section to describe automorphism group of an $m$-suspension over a line, when all weights of the $m$-suspension are greater then one and when the unique weight equals one and all the others are greater than one. The second case is a particular case of Danielewski varieties.

Let $Y=\VV(y_1^{k_1}\ldots y_m^{k_m}-P(z))$, $d=\deg P\geq 2$. We can do a linear substitution on $z$ to make $P$ be a monic polynomial with zero coefficient of $z^{d-1}$. In further we assume $P$ to be so. If $P(z)=z^d$, then $Y$ is toric (may be not normal).

\begin{lem}\label{ir}
The variety $\VV(y_1^{k_1}y_2^{k_2}\ldots y_m^{k_m}-P(z))$, $k_i\in\ZZ_{>0}$, is reducible if and only if there is $l>1$ such, that for all $i$ we have $l\mid k_i$ and there exists $Q(z)\in\KK[z]$ such that $P(z)=Q(z)^l$.
\end{lem}
\begin{proof}
Denote $d=\deg P(z)$. Suppose
$$y_1^{k_1}y_2^{k_2}\ldots y_m^{k_m}-P(z)=f(y_1,\ldots, y_m,z)g(y_1,\ldots, y_m,z),$$ 
where $f,g\neq const$. Then both $f$ and $g$ depend on $y_1$.  Indeed, if $f$ does not depend on $y_1$, then 
$$
f=f(0, y_2,\ldots, y_m,z)\mid (0y_2^{k_2}\ldots y_m^{k_m}-P(z))=P(z).
$$ 
Hence, $f=f(z)$. Therefore, $f\nmid (y_1^{k_1}y_2^{k_2}\ldots y_m^{k_m}-P(z))$. This gives a contradiction. 

Consider 
$\KK_1=\overline{\KK(y_2, \ldots, y_m, z)}$ the algebraic closure of the field $\KK(y_2, \ldots, y_m, z)$. We have 
\begin{equation}\label{snoska}
y_1^{k_1}y_2^{k_2}\ldots y_m^{k_m}-P(z)=y_2^{k_2}\ldots  y_m^{k_m}\prod_{i=1}^{k_1} (y_1-\alpha_i),
\end{equation}
where $\alpha_i\in\KK_1$ and
$$\alpha_i^{k_1}=\frac{P(z)}{y_2^{k_2}\ldots  y_m^{k_m}}.$$
Let
$$
f=h(y_2, \ldots, y_m, z) ((y_1-\alpha_{i_1})\ldots(y_1-\alpha_{i_s})), \  s<k_1.
$$
Since $f$ is a polynomial, we obtain, that $h$ is a polynomial in $y_2, \ldots, y_m, z$. Therefore, 
$$h\mid y_2^{k_2}\ldots y_m^{k_m} .$$
Hence, up to a constant, the polynomial $h$ is equal to $y_2^{r_2}\ldots y_m^{r_m}$, $r_i\leq k_i$. Denote 
$$l=\frac{k_1}{gcd (k_1,s)}.$$ 
Then $l\mid k_1$. 
The minimal degree of monomials of $f$ in $y_i$, $i\geq 2$ is 
$$\deg_{y_i}(y_2^{r_2}\ldots y_m^{r_m}\alpha_1\ldots \alpha_s)=r_i-\frac{sk_i}{k_1}\in\ZZ.$$ 
Therefore, $l\mid k_i$.

Denote $\varepsilon\in\KK$, $\varepsilon^{k_1}=1$. We can assume that $\alpha_{i+1}=\varepsilon\alpha_i$. Then putting $y_1=0$ in (\ref{snoska}) we obtain
$$
P(z)=-y_2^{k_2}\ldots  y_m^{k_m}\prod_{i=1}^{k_1} (-\alpha_i)=c y_2^{k_2}\ldots  y_m^{k_m}\alpha_1^{k_1}
$$
for some $c\in\KK$. Analogically
$$
f(0,y_2\ldots,y_m,z)=y_2^{r_2}\ldots y_m^{r_m}(-\alpha_{i_1})\ldots(-\alpha_{i_s})=b y_2^{r_2}\ldots y_m^{r_m}\alpha_1^{s}
$$
for some $b\in\KK$.

We can multiply $f$ to a constant to obtain 
$$
\left(f(0,y_2\ldots,y_m,z)\right)^{k_1}=(y_2^{r_2k_1-k_2s}\ldots y_m^{r_mk_1-k_ms})(P(z))^s.
$$
Since $f$ is a polynomial, we have $r_ik_1-k_is\geq 0$. But if $r_ik_1-k_is>0$, then $y_i\mid f(0,y_2\ldots,y_m,z)$. Hence, 
$$y_i\mid f(0,y_2\ldots,y_m,z)g(0,y_2\ldots,y_m,z)=-P(z).$$ 
A contradiction. So, $r_ik_1-k_is = 0$. We have
$$
\left(f(0,y_2\ldots,y_m,z)\right)^{k_1}=(P(z))^s.
$$

Therefore, there exists a polynomial $Q(z)$ such that 

$$
\left(f(0,y_2\ldots,y_m,z)\right)^{k_1}=(P(z))^s=Q(z)^{\mathrm{lcm}(k_1,s)}.
$$

Then $P(z)=Q^l$.
\end{proof}

\begin{lem}\label{inv}
Let  $Y=\VV(y_1^{k_1}\ldots y_m^{k_m}-P(z))$, $d=\deg P\geq 2$. Then $Y$ does not admit any nonconstant invertible functions.
\end{lem}
\begin{proof}
Firstly let $Y$ be irreducible.
Since $1$ is a $\TT$-invariant function, if $fg=1$, then $f$ and $g$ are $\TT$-semi-invariants. There are two cases:

1) $f=f(z)$. In this case if $f$ is invertible in $\KK[Y]$, then it is invertible in $\KK[z]$. That is $f$ is a constant.

2) There is $i$ such that $y_i\mid f$. We obtain, that $y_i$ is invertible. But $y_i$ admits zero volume on $Y$. 

If $Y=Y(1)\cup\ldots\cup Y(l)$ is reducible, then each irreducible component $Y(j)$ has the form  
$$
\VV\left(y_1^{\frac{k_1}{l}}\ldots y_m^{\frac{k_m}{l}}-\varepsilon^i Q(z)\right),
$$
where $\varepsilon^{l}=1$, see Lemma~\ref{ir}.

Every invertible function $f$ on $Y$ is invertible on each $Y(j)$. Therefore $f\mid_{Y(j)}=const$. Since all $Y(j)$ have a common point, $f$ is a constant.
\end{proof}

\begin{lem}\label{redu}
Let $Y=\VV(y^{k}-P(z))$ be reducible. Then $Y$ is rigid.
\end{lem}
\begin{proof}

Suppose $Y$ admits a nontrivial $(\KK,+)$-action. Since $(\KK,+)$ is a connected group, every irreducible component of $Y$ is $(\KK,+)$-invariant. 

Let $r$ be the maximal integer number admitting $R(z)\in\KK[z]$ such that $R(z)^r=P(z)$.  

Each  irreducible component of $Y$ has the following form
$$
\VV\left(y^{\frac{k}{l}}-\varepsilon(R(z))^{\frac{r}{l}}\right),\qquad \text{where} \ l=\gcd(k,r), \varepsilon^l=1.
$$
This implies that all irreducible components of $Y$ are isomorphic. Since $Y$ is not rigid, they are isomorphic to $\KK$.

But all the components have finite numbers of common points $(y,z)$, where $y~=~0$ and $P(z)~=~0$. Therefore, these points are $(\KK,+)$-stable. Since every irreducible component is isomorphic to $\KK$ and contains a $(\KK,+)$-stable point, $(\KK,+)$-action is trivial.
\end{proof}

\begin{lem}\label{cur}
Let $Y$be a curve of the form $\VV(y^k-P(z))$, $d=\deg P(z)\geq 2$, $k\geq 2$. Then $Y$ is rigid.
\end{lem}
\begin{proof}
If $Y$ is reducible, then it is rigid by Lemma~\ref{redu}.

Let $Y$ be irreducible. Irreducible curve is rigid if and only if it is not isomorphic to a line. If $P(z)$ has a multiple root, then $Y$ has a singular point and it is not isomorphic to a line. If $P(z)=a(z-z_1)\ldots(z-z_d)$ has not any multiple roots, we can compute the genus of $Y$ by Riemann-Hurwitz formula. Let us consider 
$\pi\colon \overline{Y}\rightarrow \overline{\KK}$, $\pi((y,z))=z$. We have one ramification point in the preimage of each $z_i$. The ramification index in each such point equals $k$. In the preimage of $\infty$ there are $\gcd(k,d)$ ramification points. The ramification index at each of them is equal to $\frac{k}{\gcd(k,d)}$. Therefore,
$$
(2g-2)=k(0-2)+d(k-1)+\gcd(k,d)\left(\frac{k}{\gcd(k,d)}-1\right).
$$
We obtain, 
$$g=\frac{(d-2)(k-1)}{2}+\frac{k-\gcd(k,d)}{2}=\frac{(d-1)(k-1)}{2}+\frac{1-\gcd(k,d)}{2}.$$
Assume $Y\cong \KK$. Then $g=0$. 
Since $\frac{k-\gcd(k,d)}{2}\geq 0$, if $k>1$ and $d>2$, then $g>0$. But since $g(k,d)=g(d,k)$, if  $d>1$ and $k>2$, then $g>0$.
We have $k\geq 2$ and $d\geq 2$. Therefore, there is only one opportunity $d=k=2$. But then 
$$Y\cong\VV(y^2-z^2-1)\cong\KK\setminus\{0\}\ncong\KK.$$

\end{proof}

\begin{lem}\label{r}
Let $Y=\VV(y_1^{k_1}\ldots y_m^{k_m}-P(z))$, $d=\deg P(z)\geq 2$. Assume that for all $i$ we have $k_i\geq 2$. Then $Y$ is rigid.
\end{lem}
\begin{proof}
If $m=1$ we obtain the assertion of Lemma~\ref{cur}.

Let $m>1$. Suppose $Y$ is not rigid. Then by Lemma~\ref{flza} there is a nontrivial $\TT$-homogeneous LND $\partial$. Consider $(\KK,+)$-action on $Y$ corresponding to $\partial$. Since $(\KK,+)$ is a connected group, every irreducible component $Y(j)$ of $Y$ is $(\KK,+)$-invariant. Hence, $\partial$ gives an LND on $\KK[Y(j)]$ for every irreducible component $Y(j)$. Each $Y(j)$ has the following form
$$
\VV\left(y_1^{\frac{k_1}{l}}\ldots y_m^{\frac{k_m}{l}}-\varepsilon^i Q(z)\right),
$$
where $\varepsilon^{l}=1$, see Lemma~\ref{ir}. Since $Y(j)$ is irreducible, we can apply Lemma~\ref{edin}. Hence, there is the unique~$i$ such that $\partial(y_i)\mid_{Y(j)}\neq 0$. But the index $i$ is uniquely determined by $\deg(\partial)$. Therefore, it does not depend on $j$. So, we obtain that there is unique~$i$ such that $\partial(y_i)\neq 0$.
Denote 
$\mathbb{L}=\overline{\KK(y_1,\ldots ,y_{i-1}, y_{i+1}, \ldots, y_m)}$. 
We obtain LND of the variety
$$\VV(y_i^{k_i}-P(z))\subset \mathbb{L}^2.$$ 
This contradicts to Lemma~\ref{cur}.
\end{proof}

Consider the action of one-dimensional torus on $\KK[y_1,\ldots,y_m,z]$ given by 
$$t\cdot (y_1,\ldots,y_m,z)=(t^dy_1,\ldots,y_m, t^{k_1}z).$$ 
Let $\DD\subset\KK^\times$ is the stabilizer of $Y$.  That is $\DD=\{t\in\KK^\times\mid P(t^{k_1}z)=t^{k_1d} P(z)\}$.  
We have a natural $\DD$-action on $Y$, which a priori is not effective. Let us denote by $\overline{\DD}$ the quotient group $\DD/K$, where $K$ is the kernel of $\DD$-action on $Y$. 
\begin{de}
We call $\overline{\DD}$ the {\it additional quasitorus} of $Y$. 
\end{de}
Also we use the following notation $\widehat{\DD}=\overline{\DD}/(\overline{\DD}\cap \HH)$.

\begin{lem}\label{pl}
If $P(z)=z^d$, then $\DD\cong\KK^\times$, $\overline{\DD}\cong\KK^\times$, and $\widehat{\DD}\cong\KK^\times$. If $P(z)\neq z^d$ and $v$ is the maximal integer such that $P(z)~=~z^uQ(z^v)$ for some nonnegative $u$ and $v\leq d$, then $\DD\cong\ZZ_{vk_1}$, $\overline{\DD}\cong\ZZ_{\mathrm{lcm}(k_1,v)}$, and $\widehat{\DD}\cong\ZZ_v$. In particular, if $v=1$, then $\DD=\overline{\DD}\cong\ZZ_{k_1}\subset \HH$ and $\widehat{\DD}$ is trivial. 
\end{lem}
\begin{proof}
Suppose for some $t\in\KK\setminus\{0\}$ we have $P(t^{k_1}z)=t^{k_1d} P(z)$. Then the mapping $z\mapsto t^{k_1}z$ permutes the roots of $P$. If $t$ is not a root of unity, this implies that all roots of $P$ are equal to zero, that is $P(z)=z^d$. If $t^{k_1}$ is a primitive $a$-root of unity, then for every nonzero root $z_0$ of P, we obtain that $z_0, t^{k_1}z_0, \ldots t^{(a-1)k_1}z_0$ are roots of $P$. That is $P(z)=z^bQ(z^a)$. 

The group $\DD$ is a subgroup in $\KK^\times$. Therefore, $\DD$ is isomorphic either to $\KK^\times$ or to $\ZZ_a$. If $\DD\cong\KK^\times$, we have $P(z)=z^d$. 
If $P(z)=z^uQ(z^v)$, then $\DD\supset\ZZ_{k_1v}$. But if $\DD\neq\ZZ_{k_1v}$, then there is $t$ in $\DD$ such that $t^{k_1}$ is not a $v$-root of unity. 
Therefore, there is $a>v$ such that $P(z)=z^bQ(z^a)$ for some nonnegative $b$. 

All the other assertions of the lemma are easy.
\end{proof}

\section{Automorphisms of $m$-suspensions over a line with weights $\geq 2$}\label{no}

In this section we give an explicit list of generators of  automorphism group of an $m$-suspension over a line, when all weights of the $m$-suspension is greater then one. We obtain a decomposition of the automorphism group to a semidirect  product of its subgroups. Using this we obtain criteria of automorphism group of such an $m$-suspension to be commutative and solvable.

\begin{theor}\label{vtor}
Let $Y=\VV(y_1^{k_1}\ldots y_m^{k_m}-P(z))$, $d=\deg P(z)\geq 2$, and for all $i$ we have $k_i\geq 2$. We assume that $P(z)$ is a monic polynomial with zero coefficient of $z^{d-1}$. Then the group $\mathrm{Aut}(Y)$ is generated by 

1) the proper quasitorus $\HH$ of $Y$;

2) the symmetric group $S(Y)$ of Y;

3) the additional quasitorus $\overline{\DD}$ of $Y$;

4) if $m=1$ and $k_1=d=2$, then mapping 
$$
\begin{pmatrix}
y_1\\
z
\end{pmatrix}\mapsto
\begin{pmatrix}
a&b\\
b&a
\end{pmatrix}
\begin{pmatrix}
y_1\\
z
\end{pmatrix}, \qquad
a^2-b^2=1,
$$
should be added.
\end{theor}
\begin{proof}
By Lemma~\ref{r}, $Y$ is rigid. Hence, by~\cite[Theorem~2.1]{AG} (see also \cite[Corollary~3.4]{FZ}), $\Aut(Y)$ contains the unique maximal torus $\overline{T}$. This implies that every automorphism  normalizes the $\overline{T}$-action on $Y$. 

Suppose $Y$ does not admit an effective action of $m$-dimensional torus. Then $\overline{T}=\TT$. Let $f$ be a $\TT$-semi-invariant of nonzero weight. Then there is an index $i$ such that $y_i\mid f$. Lemma~\ref{inv} implies, that $Y$ does not admit any nonconstant invertible functions. Then every indecomposable $\TT$-semi-invariant of a nonzero weight has the form $\lambda y_i$, $\lambda\in\KK\setminus\{0\}$.  
Let $\psi \in\Aut(\KK[Y])$. Then $\psi$ maps indecomposable $\TT$-semi-invariants onto indecomposable $\TT$-semi-invariants. Therefore there are a permutation $\sigma\in\mathrm{S}_m$ and $\lambda_1,\ldots,\lambda_m\in\KK\setminus\{0\}$  such that $\psi(y_j)=\lambda_j y_{\sigma(j)}$. 
We have, that
$$\psi\left(y_1^{k_1}\ldots y_m^{k_m}\right)=\left(\prod_{j=1}^m\lambda_j^{k_j}\right)y_{\sigma(1)}^{k_1}\ldots y_{\sigma(m)}^{k_m}$$
is a $\TT$-invariant. Therefore, for all $i$ we have $k_i=k_{\sigma(i)}$. 

There is an automorphism $\tau$ in the semidirect product $S(Y)\rightthreetimes \TT$ such that 
$$
\forall j\in\{2,\ldots,m\}\ \text{we have}\  (\tau\circ\psi)(y_j)=y_j, (\tau\circ\psi)(y_1)=\mu y_1,\qquad \mu\in\KK\setminus\{0\}.
$$ 

Since for every $\alpha\in\Aut(\KK[Y])$ we have $\alpha(y_j)=\lambda_j y_{\sigma(j)}$ and for all $j$ it occurs $k_j=k_{\sigma(j)}$, we obtain that~$\HH$ is a normal subgroup of $\Aut(\KK[Y])$. Therefore, $\KK[Y]^{\HH}=\KK[z]$ is invariant under every $\alpha\in\Aut(\KK[Y])$.

Since $\tau\circ\psi$ is an isomorphism, we have $(\tau\circ\psi)(z)=az+b$ for some $a\in\KK\setminus\{0\}$ and $b\in\KK$. 
We obtain $P(az+b)=\mu^{k_1}P(z)$. Therefore, the mapping $z\mapsto az+b$ permutes the roots of $P(z)$. Since $P(z)$ has zero coefficient of $z^{d-1}$, the sum of roots of $P(z)$ is equal to zero. Therefore, $b=0$. So $P(az)=\mu^{k_1}P(z)$. This implies $(\tau\circ\psi)\in\overline{\DD}$.

If $Y$ admits an effective action of $m$-dimensional torus.  Then $P(z)=z^d$. Indeed, $\TT\subset\overline{T}$. That is every element $t\in\overline{T}$ commutes with $\TT$. Therefore, $t\cdot y_i=\lambda_i y_i$. As in the nontoric case we have $t\cdot z=az+b$. Since $\overline{T}$ is a connected group, it can not nontrivially permute the roots of $P(z)$.  This implies $b=0$ and  $P(z)=z^d$. 

Let us consider a set $W$ consisting of all $\overline{T}$-weights~$w$ such that $\KK[Y]_w\neq\{0\}$. It is easy to see that $W$  is a monoid generated by the weights of all $y_i$ and the weight of $z$. Indeed, in every homogeneous component there is a monomial $y_1^{a_1}\ldots y_m^{a_m}z^b$, which is not divisible by $y_1^{k_1}\ldots y_m^{k_m}-P(z)$, and therefore, it is nonzero.  We call $W$
the {\it weight monoid} of $Y$. Let $V=W\otimes \mathbb{Q}$ be a $\mathbb{Q}$-vector space spanned by $W$. Linear combinations of elements of $W$ with nonnegative coeffisients form a cone, which is called the {\it weight cone}.
It is easy to see, that $\overline{T}$ - action is {\it pointed}, that is the weight cone does not contain any lines. The weight cone is generated by weights of $y_1,\ldots,y_m$. If $m\neq 1$, this is the unique minimal set of generators of the weight cone consisting of indecomposable weights. Hence, there are $\sigma\in\mathrm{S}_m$ and $\lambda_1,\ldots,\lambda_m\in\KK\setminus\{0\}$ such that $\psi(y_j)=\lambda_j y_{\sigma(j)}$.

Let $w_1,w_2\in W$. We say that $w_1\leq w_2$ if  there exists $w_3\in W$ such that $w_1+w_3=w_2$. Let $U$ be the set of weights of homogeneous elements in $\KK[Y]\setminus\KK[y_1,\ldots,y_m]$. It is easy to see, that the weight of $z$ is the unique minimal nonzero element in $U$. Therefore, the weight of $\psi(z)$ equals the weight of $z$. Hence, the weight of $\psi(y_1^{k_1}\ldots y_m^{k_m})$ equals the weight of $y_1^{k_1}\ldots y_m^{k_m}$.  This follows, that for all $j$ we have $k_j=k_{\sigma(j)}$.
Hence, there is $\tau\in S(Y)$ such that $(\tau\circ\psi)(y_i)=\lambda_i y_i$. Then $(\tau\circ\psi)(z)=az$. Therefore, there exists $\xi\in\overline{\DD}$ such that $(\xi\circ\tau\circ\psi)(z)=z$, $(\xi\circ\tau\circ\psi)(y_i)=\mu_i y_i$. Hence, $(\xi\circ\tau\circ\psi)\in\HH$.

If $m=1$, we have $Y=\VV(y_1^{k_1}-z^d)$. Then $\overline{T}=\overline{\DD}$. Functions $y_1$ and $z$ are $\overline{T}$-semi-invariants. If $k_1<d$, then the weight of $z$ is less then the weight of $y_1$. That is $\psi(z)=\lambda z$. We can take $\tau\in\overline{\DD}$ such that $(\tau\circ\psi)(z)=z$. Then $\tau\circ\psi\in\HH$. If $k_1<d$, then the weight of~$y_1$ is less then the weight of $z$. That is $\psi(y_1)=\lambda y_1$. Then $\psi(z)=\mu z$. Therefore, again there is $\tau\in\overline{\DD}$ such that $\tau\circ\psi\in\HH$.

If $k_1=d$, then the weights of $y_1$ and $z$ are equal. Since it is the unique indecomposable weight, the weights of $\psi(z)$ and $\psi(y_1)$ are the same. Therefore, 
$$\psi(z)=az+by_1,\qquad \psi(y_1)=a'z+b'y_1.$$ We have $\psi(y_1^d-z^d)=\mu(y_1^d-z^d)$, $\mu\in\KK\setminus\{0\}$. Suppose $d\geq 3$. We obtain
$$
\begin{cases}
a^d-a'^d=-\mu;\\
a^{d-1}b-a'^{d-1}b'=0;\\
a^{d-2}b^2-a'^{d-2}b'^2=0;\\
\vdots\\
b^d-b'^d=\mu.
\end{cases}
$$
Let $a'\neq 0$. Then from the second equation we obtain $b'=\frac{a^{d-1}b}{a'^{d-1}}$. From the third equation we have 
$$a^{d-2}b^2-a'^{d-2}\left(\frac{a^{d-1}b}{a'^{d-1}}\right)^2=0.$$
Hence, $a^{d-2}b^2(a'^{d}-a^d)=0.$ That is $a^{d-2}b^2\mu=0$. Hence, $a=0$ or $b=0$. Analogically if $a\neq 0$, then $a'=0$ or $b'=0$. And so on. 
Therefore, we have $a=b'=0$ or $b=a'=0$. Hence, $\psi$ can be obtained by composition of a $\overline{\DD}$-element, an $\HH$-element, and may be swapping $y_1\leftrightarrow z$. But since $y_1^d=z^d$, we have that $y_1=\varepsilon z$ for some $d$-root of unity $\varepsilon$. Therefore, swapping $y_1\leftrightarrow z$ can be obtained as a composition of a $\overline{\DD}$-element and an $\HH$-element.

Now let $m=1$, $d=k_1=2$. Then we obtain 
$$
\begin{cases}
a^2-a'^2=-\mu;\\
ab-a'b'=0;\\
b^2-b'^2=\mu.
\end{cases}
$$
Up to element of $\overline{T}$, we can assume $\mu=1$. It is easy to see, that this gives $a^2-b^2=1$ and $a'=\pm b$, $b'=\pm a$. Up to element of $\HH\cong\ZZ_2$, we have $a^2-b^2=1$, $a'=b$, $b'=a$. 
\end{proof}

Note that $\HH$ and $\overline{\DD}$ commute. Therefore, the subgroup in $\Aut(Y)$ generated by $\HH$ and $\overline{\DD}$ is isomorphic to $(\HH\times\overline{\DD})/(\HH\cap\overline{\DD})$. So we obtain the following corollary. 

\begin{cor}\label{ct}
Let $Y=\VV(y_1^{k_1}\ldots y_m^{k_m}-P(z))$, $d=\deg P(z)\geq 2$, and for all $i$ we have $k_i\geq 2$. We assume, that $P(z)$ is a monic polynomial with zero coefficient of $z^{d-1}$. Suppose $Y\neq\VV(y_1^{2}-z^2)$. Then the group $\Aut(Y)$ is isomorphic to a semidirect product 
$$S(Y)\rightthreetimes ((\HH\times\overline{\DD})/(\HH\cap\overline{\DD})).$$
\end{cor}

\begin{re}
We obtain, that $\Aut(Y)$ is a linear algebraic group. If $Y\neq\{y_1^{2}=z^2\}$, then the neutral component of $\Aut(Y)$ is a torus.
\end{re}

Corollary~\ref{ct} and Lemma~\ref{pl} imply the following corollary. 
\begin{cor}
Let $Y=\VV(y_1^{k_1}\ldots y_m^{k_m}-P(z))$, $d=\deg P(z)\geq 2$, and for all~$i$ we have $k_i\geq 2$. We assume, that $P(z)$ is a monic polynomial with zero coefficient of $z^{d-1}$. Suppose that there is no $v>1$ such that $P(z)=z^uQ(z^v)$. Then the group $\Aut(Y)$ is isomorphic to a semidirect product $S(Y)\rightthreetimes \HH$.
\end{cor}

Let us give some more corollaries from Theorem~\ref{vtor}.

\begin{cor}\label{cb}
Let $Y=\VV(y_1^{k_1}\ldots y_m^{k_m}-P(z))$, $d=\deg P(z)\geq 2$, and for all~$i$ we have $k_i\geq 2$. If $Y\neq\{y_1^{2}=z^2\}$, then the group $\Aut(Y)$ is commutative if and only if all $k_i$ are different. If $Y=\{y_1^{2}=z^2\}$, then the group $\Aut(Y)$ is not commutative.
\end{cor}
\begin{proof}
Suppose that $k_i=k_j$. Then we can consider the subgroup isomorphic to $(\KK^\times,\cdot)$ acting by $t\cdot y_i=ty_i$, $t\cdot y_j=t^{-1}y_j$ and the subgroup isomorphic to $\ZZ_2$ swapping $y_i$ and~$y_j$. Then $\Aut(Y)$ contains a subgroup isomorphic to $\ZZ_2\rightthreetimes\KK^\times$, where nonzero element of $\ZZ_2$ acts on $\KK^\times$ by multiplication to $-1$. Therefore, $\Aut(Y)$ is not commutative. 

Let $Y\neq\{y_1^{2}=z^2\}$. If all $k_i$ are different, then $S(Y)$ is trivial. Therefore, by Corollary~\ref{ct}, $\Aut(Y)$ is a quasitorus. 

An explicit description of $\Aut(\{y_1^{2}=z^2\})$ was obtained in proof of Theorem~\ref{vtor}. One can check that it is not commutative.
\end{proof}

If $\Aut(Y)$ is commutative, then it is a quasitorus. Let us investigate, when it is a torus.

\begin{cor}\label{toraut}
Let $Y=\VV(y_1^{k_1}\ldots y_m^{k_m}-P(z))$, $d=\deg P(z)\geq 2$, and for all~$i$ we have $k_i\geq 2$. We assume, that $P(z)$ is a monic polynomial with zero coefficient of $z^{d-1}$. Then the group $\Aut(X)$ is a torus if and only if all $k_i$ are different and one of the following occurs 

1) $P(z)=z^d$ and $\gcd(k_1,\ldots,k_m,d)=1$; 

2) there is no $v>1$ such, that $P(z)=z^uQ(z^v)$, and $\gcd(k_1,\ldots,k_m)=1$.
\end{cor}
\begin{proof}
If there are indices $i$ and $j$ such that $k_i=k_j$, then $\Aut(Y)$ is not commutative. Further we assume, that all $k_i$ are different.
   
Let $P(z)=z^d$. If $\gcd(k_1,\ldots,k_m,d)=r>1$, then, by Lemma~\ref{ir}, $Y$ is reducible. There is automorphism $\psi$ of the form $y_i\mapsto y_i$, $z\mapsto \varepsilon z$, where $\varepsilon^d=1$, which nontrivially permutes irreducible components. But a torus is a connected group, which can not give a nontrivial permutation of irreducible components. Now let
$\gcd(k_1,\ldots,k_m,d)=1$. Denote $\gcd(k_1,\ldots,k_m)=q$. Then $\gcd(q,d)=1$. By Theorem~\ref{vtor}, $\Aut(Y)$ is generated by $\HH~\cong~(\KK^\times)^{m-1}\times\ZZ_{q}$ and $\overline{\DD}\cong\KK^\times$. There is $t_0\in\KK^\times$ such that $t_0^{k_1}=1$ and $t_0^{\frac{k_1}{q}}$ is a primitive $q$-root of unity. Therefore, $t_0^d$  is a primitive $q$-root of unity. Consider $\DD$-action on $Y$ given~by
$$t\cdot(y_1,\ldots,y_m,z)=(t^dy_1,y_2,\ldots,y_m,t^{k_1}z).$$
If we take $t=t_0$, then $z\mapsto z$, $y_1\mapsto t_0^dy_1$. This gives the action of torsion subgroup of $\HH$. Therefore, $\Aut(Y)$ is generated by $\overline{\DD}$ and $\TT$. Thus, $\Aut(Y)$ is a torus.

If $P(z)\neq z^d$ and $P(z)=z^uQ(z^v)$ for some $2\leq v\leq d$, then $\overline{\DD}$ is a finite nontrivial group, which nontrivial acts on $z$. Since $\HH$ preserves $z$, the $\Aut(Y)$-orbit of $z$ is finite. That is $\Aut(Y)$ is not a torus.

If $P(z)\neq z^d$ and $\gcd(k_1,\ldots,k_m)\neq 1$, then $\HH$ is not connected and $\overline{\DD}$ is finite. That is $\Aut(Y)$ is not a torus.

If there is no $v>1$ such that $P(z)=z^uQ(z^v)$ and $\gcd(k_1,\ldots,k_m)=1$, then $\HH=\TT$ is a torus and $\overline{\DD}\subset\HH$, see Lemma~\ref{pl}. 
Therefore, $\Aut(Y)=\TT$ is a torus.
\end{proof}

\begin{cor}\label{sb}
Let $Y=\VV(y_1^{k_1}\ldots y_m^{k_m}-P(z))$, $d=\deg P(z)\geq 2$, and for all~$i$ we have $k_i~\geq~2$. Then the group $\Aut(Y)$ is not solvable if and only if there are five indices $i_1,i_2,i_3,i_4,i_5\in\{1,\ldots,m\}$ such that $k_{i_1}=k_{i_2}=k_{i_3}=k_{i_4}=k_{i_5}$.
\end{cor}
\begin{proof}
If such indices exist, then $S(Y)$ contains $\mathrm{S}_5$. Hence, $\Aut(Y)$ is not solvable.

If there are no such five indices, then $S(Y)$ is solvable. Therefore, if $Y\neq\VV(y_1^{2}-z^2)$, by Corollary~\ref{ct}, $\Aut(Y)$ is a semidirect product of a solvable group and a commutative one.

Now let $Y=\VV(y_1^{2}-z^2)$. Then $Y$ has two irreducible components $L_1$ and $L_2$. Each automorphism either preserves this components or swaps them. We obtain a homomorphism $\Aut(Y)\rightarrow\mathrm{S}_2$. Let $G$ be the kernel of this homomorphism. Then there is a homomorphism $f\colon G\rightarrow \Aut(L_1)$. The kernel $\Ker(f)$ is contained in $\Aut(L_2)$. The automorphism group of a line is isomorphic to $\KK^\times\rightthreetimes  \KK$, therefore, it is solvable. Hence, $\Aut(Y)$ is solvable.
\end{proof}

\begin{ex}
Let $Y=\VV(y_1^2y_2^3-z^4-az^2-bz)$. 

If $a=b=0$, then $Y$ satisfies condition 1) from Corollary~\ref{toraut}. Therefore,  $\Aut(Y)\cong (\KK^\times)^2$.

If $a=0$ and $b\neq 0$, then $z^4+bz=z(z^3+b)$. Hence, $\Aut(Y)\cong \ZZ_3\times(\KK^\times)$. 

If $a\neq 0$ and $b=0$, then $z^4+az^2=z^2(z^2+a)$. Therefore,
$\Aut(Y)\cong \ZZ_2\times(\KK^\times)$. 

If $a\neq 0$ and $b\neq 0$, then $Y$ satisfies condition 2) from Corollary~\ref{toraut}. Hence, $\Aut(Y)\cong \KK^\times$.
\end{ex}

Let us give an example, where $\Aut(Y)\ncong S(Y)\rightthreetimes (\HH\times\widehat{\DD})$ 
\begin{ex}
Let $Y=\VV(y_1^4y_2^2-z^6-1)$. Then $\HH\cong\KK^\times\times\ZZ_2$, $\overline{\DD}\cong\ZZ_{12}$, $S(Y)$ is trivial. 
The action of $\HH$ is given by 
$$(t,\varepsilon)\cdot (y_1,y_2,z)=(t y_1, t^{-2}\varepsilon y_2, z),\qquad \varepsilon=\pm 1.$$
The action of $\overline{\DD}$ is given by 
$$\eta\cdot (y_1,y_2,z)=(\eta^3 y_1, y_2, \eta^2 z), \qquad \eta^{12}=1.$$
We have $\HH\cap\overline{\DD}\cong\ZZ_2$ and $(\HH\times\overline{\DD})/(\HH\cap\overline{\DD})\cong \KK^\times\times \ZZ_{12}$. 

Therefore, 
$\Aut(X)\cong\KK^\times\times \ZZ_{12}\ncong\KK^\times\times \ZZ_{2}\times\ZZ_6$.
\end{ex}

\section{Automorphisms of $m$-suspensions over a line with the unique weight equal to one}\label{oo}

During this section let $Y=\VV(y_1y_2^{k_2}\ldots y_m^{k_m}-P(z))$ be an $m$-suspension with the unique weight equal to one and all other weights greater than one. Such variety is a particular case of a Danielewski variety. We give an explicit lists of generators of automorphism group and describe $\Aut(Y)$ as a semidirect product of its subgroups.

The variety $Y$ admits a nontrivial LND.
\begin{de}\label{cd}
Let us define the {\it canonical LND} $\widehat{\partial}$ on $Y$ by
$$
\widehat{\partial} (y_1)=P'(z), \qquad \widehat{\partial}(z)=y_2^{k_2}\ldots y_m^{k_m}, \qquad\widehat{\partial}(y_j)=0, j\geq 2.
$$ 
\end{de}

\begin{lem}\label{mlsul}
Let $Y=\VV(y_1y_2^{k_2}\ldots y_m^{k_m}-P(z))$, $d=\deg P(z)\geq 2$. Assume for all $i$ we have $k_i\geq 2$. Then 

1) $\mathrm{ML}(Y)=\KK[y_2, \ldots, y_m]$.

2) Every LND on $\KK[Y]$ has the form $h(y_2,\ldots,y_m)\widehat{\partial}$,

3) For every subgroup $\Lambda\cong(\KK^\times,\cdot)\subset\rm{Aut}(X)$, $\varphi\in\Lambda$, $j\geq 2$, there exists $\lambda\in\KK\setminus\{0\}$ such that $\varphi^*(y_j)=\lambda y_j$. 

4) For every $\psi\in\rm{Aut}(\KK[Y])$ and $j\geq 2$ there exist $\lambda_j\in\KK\setminus \{0\}$ and $p\in\{2,\ldots,m\}$ such that $\psi(y_j)=\lambda_jy_p$.

\end{lem}
\begin{proof} {\bf 1)} By Lemma~\ref{ir}, $Y$ is irreducible.
If $m=1$, we have $Y\cong\KK$, hence $\mathrm{ML}(Y)=\KK$.

Let $m\geq 2$. Then $Y_1$ is rigid by Lemma~\ref{r}. Let us fix $j\in\{2,\ldots,m\}$. Consider $\ZZ$-grading $\mathfrak{h}_{1j}$ of $\KK[Y]$, $\deg y_1=-k_j$, $\deg y_j=1$, degrees of all the other variables equal zero. Let $\delta$ be a $\ZZ$-homogeneous LND. Then by Lemma~\ref{osn}(3) $\deg \delta>0$. Let $\partial$ be a nontrivial LND of~$\KK[Y]$. Then $\partial=\sum_{i=l}^k\partial_i$, where $\partial_i$ is homogeneous of degree~$i$. By Lemma~\ref{fl}(1), $\partial_l$ is a nontrivial LND of $\KK[Y]$. Hence, $l>0$. Therefore, for all $i$ we obtain $\deg \partial_i(y_j)>0$, that is $y_j\mid\partial_i(y_j)$. Hence, $y_j\mid\partial(y_j)$. This implies $\partial(y_j)=0$, i.e. $y_j\in\mathrm{ML}(Y)$.

Let us consider the canonical LND $\widehat{\partial}$. We have $\Ker \widehat{\partial}=\KK[y_2,\ldots,y_m]$. Then we obtain $\mathrm{ML}(Y)~\subset~\KK[y_2, \ldots, y_m]$.
Therefore, $\mathrm{ML}(Y)=\KK[y_2, \ldots, y_m]$.

{\bf 2)}  
Again let us consider the  $\ZZ$-grading $\mathfrak{h}_{1j}$ of $\KK[Y]$. For every nontrivial LND $\partial=\sum_{i=l}^k\partial_i$ we have, that $\partial_l$ is a nontrivial LND. In 1) we have proved, that for all $j\geq 2$ we have $\partial_l(y_j)=0$. This implies $\partial_l(y_1)\neq 0$. Therefore, $y_1\nmid\partial_l(y_1)$. Hence, $\deg(\partial_l(y_1))\geq 0$. That is $l\geq k_j$. Therefore, $y_j^{k_j}\mid\partial(z)$. So, $(y_2^{k_2}\ldots y_m^{k_m})\mid\partial(z)$. We have $\partial(z)=h\widehat{\partial}(z)$ for some $h\in\KK[Y]$. Obtain 
$$
\partial(y_1)=\partial\left(\frac{P(z)}{y_2^{k_2}\ldots y_m^{k_m}}\right)=\frac{P'(z)\partial(z)}{y_2^{k_2}\ldots y_m^{k_m}}=\frac{P'(z)h\widehat{\partial}(z)}{y_2^{k_2}\ldots y_m^{k_m}}=
h\widehat{\partial}(y_1).
$$ 

Since for all $ i\geq 2$ we have $\partial(y_i)=0=h\widehat{\partial}(y_i)$, we obtain 
$\partial=h\widehat{\partial}$. Since $\partial$ is an LND, $h\in\Ker\widehat{\partial}=\KK[y_2, \ldots, y_m]$, see~\cite[Principle~7]{Fr}. We obtain, that $Y$ is almost rigid. 

{\bf 3)} Let $\delta$ be the semisimple derivation corresponding to $\Lambda$. Then $\delta=\sum_{i=l}^k\delta_i$.  If $l<0$, then $\delta_l$ is an LND. Hence, $l\geq 0$. This implies 
$y_j\mid\delta_i(y_j)$. Therefore, $y_j\mid\delta(y_j)$. By Lemma~\ref{semisimp},  there exists $\lambda\in\KK\setminus\{0\}$ such that $\varphi^*(y_j)=\lambda y_j$.

{\bf 4)} Let $\psi\in\rm{Aut}(\KK[Y])$. Then $\psi(y_j)$, $j\geq 2$, is an indecomposable element of $\rm{ML}(Y)$ which is semi-invariant with respect to all 
$$\Lambda\cong(\KK^\times,\cdot)\subset\rm{Aut}(Y).$$ 
In particularly, $\psi(y_j)$ is $\TT$-semi-invariant. Therefore, there exist $\lambda_j$ in $\KK\setminus \{0\}$ and $p$ in $\{2,\ldots,m\}$ such that $\psi(y_j)=\lambda_jy_p$.
\end{proof}

For the variety $Y=\VV(y_1y_2^{k_2}\ldots y_m^{k_m}-P(z))$ the proper quasitorus $\HH$ coincides with the torus $\TT$. Let us call it the {\it proper torus} of $Y$. We have $\overline{\DD}=\DD$. If $P(z)=z^d$, then $\DD$ is a torus acting by $t\cdot(y_1, y_2,\ldots,y_m,z)=(t^d y_1, y_2,\ldots, y_m, tz)$. If $P(z)=z^uQ(z^v)\neq z^d$, and $v$ is maximal with this property, then $\DD\cong\ZZ_v$ acts by $\varepsilon\cdot(y_1,\ldots,y_m,z)=(y_1,\ldots,y_m, \varepsilon z)$, where $\varepsilon^v=1$, see Lemma~\ref{pl}.

If $m=1$, then $Y\cong\KK$. Further we assume $m\geq 2$.

\begin{prop}\label{drep}
Let $Y=\VV(y_1y_2^{k_2}\ldots y_m^{k_m}-P(z))$, $d=\deg P(z)\geq 2$, $m\geq 2$. Assume that for all~$i\geq 2$ we have $k_i\geq 2$. Let $P(z)$ be a monic polynomial with zero coefficient of~$z^{d-1}$. Then the group $\Aut(Y)$ is generated by 

1) the proper torus $\TT$ of $Y$;

2) the big unipotent subgroup $\UU(\widehat{\partial})$ corresponding to the canonical LND on $Y$;

3) the additional quasitorus $\DD$ of $Y$;

4) the symmetric group $S(Y)$ of $Y$.
\end{prop}
\begin{proof}
By Lemma~\ref{mlsul} (2) every LND on $\KK[Y]$ has the form $h(y_2,\ldots,y_m)\widehat{\partial}$. Therefore, for every LND~$\partial$ we have $\partial^2(z)=0$. Let $\psi\in\Aut(\KK[Y])$. 
Then for every LND~$\partial$ we have $\partial^2(\psi(z))=0$.  Hence, $\widehat{\partial}^2(\psi(z))=0$. This implies $\psi(z)=a(y_2,\ldots,y_m)z+b(y_2,\ldots,y_m)$. Since $\psi$ is isomorphism we obtain $a\in\KK^\times$.
We have 
\begin{multline*}
\psi(y_1y_2^{k_2}\ldots y_m^{k_m})=\psi(P(z))=P(az+b(y_2,\ldots, y_m))=\\
=a^dP(z)+da^{d-1}b(y_2,\ldots, y_m)z^{d-1}+\Delta(y_2,\ldots, y_m, z),
\end{multline*}
where $\deg_z\Delta(y_2,\ldots, y_m, z)\leq d-2$.

By Lemma~\ref{mlsul} (4),  for every $\psi\in\Aut(\KK[Y])$ there is a permutation $\sigma\in \rm{S}_{m-1}$ and $\lambda_2,\ldots,\lambda_m\in\KK\setminus \{0\}$ such that $\psi(y_i)=\lambda_iy_{\sigma(i)}$, $i\geq 2$.

We have
\begin{multline*}
\psi(y_1)=\frac{\psi(y_1y_2^{k_2}\ldots y_m^{k_m})}{\psi(y_2^{k_2}\ldots y_m^{k_m})}=\frac{P(az+b(y_2,\ldots,y_m))}{\lambda_2^{k_2}\ldots \lambda_m^{k_m}y_{\sigma(2)}^{k_2}\ldots y_{\sigma(m)}^{k_m}}=\\
=\frac{a^dP(z)+da^{d-1}b(y_2,\ldots, y_m)z^{d-1}+\Delta(y_2,\ldots, y_m, z)}{\lambda_2^{k_2}\ldots \lambda_m^{k_m}y_{\sigma(2)}^{k_2}\ldots y_{\sigma(m)}^{k_m}}=\\
=\frac{a^dy_1y_2^{k_2}\ldots y_m^{k_m}+da^{d-1}b(y_2,\ldots, y_m)z^{d-1}+\Delta(y_2,\ldots, y_m, z)}{\lambda_2^{k_2}\ldots \lambda_m^{k_m}y_{\sigma(2)}^{k_2}\ldots y_{\sigma(m)}^{k_m}}.
\end{multline*}
In the last quotient the numerator does not contain $z^k$, $k\geq d$. Therefore, we obtain, that in $\KK[y_1,\ldots,y_m,z]$ the following accurs
\begin{equation}\label{eq}
\prod_{i=2}^my_{\sigma(i)}^{k_i}\mid(a^dy_1y_2^{k_2}\ldots y_m^{k_m}+da^{d-1}b(y_2,\ldots, y_m)z^{d-1}+\Delta(y_2,\ldots, y_m, z)).
\end{equation}
Therefore, 
$$
y_{\sigma(2)}^{k_2}\ldots y_{\sigma(m)}^{k_m}\mid b(y_2,\ldots, y_m).
$$
If we put $z=0$ in (\ref{eq}), we obtain 
$$
\prod_{i=2}^my_{\sigma(i)}^{k_i}\mid(a^dy_1y_2^{k_2}\ldots y_m^{k_m}+P(0)(1-a^d)+b(y_2,\ldots,y_m)g(y_2,\ldots,y_m)).
$$
This implies $P(0)(1-a^d)=0$ and $\prod_{i=2}^my_{\sigma(i)}^{k_i}\mid y_1y_2^{k_2}\ldots y_m^{k_m}$. Therefore, $k_{\sigma(i)}=k_i$. 
Let us consider the automorphism  $\alpha\in S(Y)$ given by $\alpha(y_1)=y_1,\ \alpha(z)=z,\ \alpha(y_j)=y_{\sigma^{-1}(j)}.$ 
Then $(\alpha\circ\psi)(y_j)=\lambda_jy_j$. 

We can take $\tau\in \TT$ such that 
$$
(\tau\circ\alpha\circ\psi)(y_j)=y_j,\ j\geq 2,\qquad (\tau\circ\alpha\circ\psi)(z)=az+b'(y_2,\ldots, y_m),$$
where $b'(y_1,\ldots, y_m)=b(\tau(y_1),\tau(y_{\sigma^{-1}(2)}),\ldots,\tau(y_{\sigma^{-1}(m)}))$.

Since 
$$
y_{2}^{k_2}\ldots y_m^{k_m}\mid b'(y_2,\ldots, y_m),
$$
we can denote $b'(y_2,\ldots, y_m)=y_{2}^{k_2}\ldots y_m^{k_m}h(y_2,\ldots,y_m)$.
Denote $\omega=\exp(-h\widehat{\partial})$. Then $(\omega\circ\tau\circ\alpha\circ\psi)=\nu$, $\nu(y_j)=y_j$, $j\geq 2$, $\nu(z)=az$.

We have $y_2^{k_2}\ldots y_m^{k_m}\mid P(az)$. Therefore, $P(az)=a^dP(z)$. Hence, $\nu(y_1)=a^dy_1$. 
That is~$\nu\in\DD$.
\end{proof}

For the variety $Y=\VV(y_1y_2^{k_2}\ldots y_m^{k_m}-P(z))$, we have $\TT\cap\DD=\{\mathrm{id}\}$. Therefore, we obtain the following theorem.

\begin{theor}\label{vtt}
Let $Y=\VV(y_1y_2^{k_2}\ldots y_m^{k_m}-P(z))$, $d=\deg P(z)\geq 2$, $m\geq 2$. Assume that for all~$i\geq 2$ we have $k_i\geq 2$. Let $P(z)$ be a monic polynomial with zero coefficient of $z^{d-1}$. Let $\TT$ be the proper torus of $Y$, $\DD$ be the additional quasitorus of $Y$, $S(Y)$ be the symmetric group of $Y$, and $\UU(\widehat{\partial})$ be the big unipotent subgroup corresponding to the canonical LND $\widehat{\partial}$ on $Y$. Then the automorphism group of $Y$ is isomorphic~to
$$S(Y)\rightthreetimes\left( \left(\TT\times\DD\right)\rightthreetimes \UU(\widehat{\partial})\right).$$
\end{theor}

\begin{cor}
Let $Y=\VV(y_1y_2^{k_2}\ldots y_m^{k_m}-P(z))$, $d=\deg P(z)\geq 2$, $m\geq 2$. Assume that for all~$i\geq 2$ we have $k_i\geq 2$. If all $k_i$ are different and there is no $v>1$ such that $P(z)=z^uQ(z^v)$, then $\Aut(Y)$ is a semidirect product of $m-1$-dimensional torus and a commutative group $\UU(\widehat{\partial})$.
\end{cor}

\begin{cor}
Let $Y=\VV(y_1y_2^{k_2}\ldots y_m^{k_m}-P(z))$, $d=\deg P(z)\geq 2$, $m\geq 2$. Assume that for all~$i\geq 2$ we have $k_i\geq 2$. Then $\Aut(X)$ is not commutative.
\end{cor}
\begin{proof}
Let us consider $t_0\in\TT$, acting by $t_0\cdot(y_1,y_2,\ldots, y_m,z)=((-1)^{k_2}y_1,-y_2,\ldots, y_m,z)$ 
and
$\alpha=\exp\left(y_2^{k_2-1}\widehat{\partial}\right)$. It is easy to see, that these elements of $\Aut(Y)$ do not commute. 
\end{proof}

Analogically to Corollary~\ref{sb}, we obtain the following corollary.
\begin{cor}\label{sold}
Let $Y=\VV(y_1y_2^{k_2}\ldots y_m^{k_m}-P(z))$, $d=\deg P(z)\geq 2$, and for all~$i$ we have $k_i~\geq~2$. Then the group $\Aut(Y)$ is not solvable if and only if there are five indices $i_1,i_2,i_3,i_4,i_5\in\{1,\ldots,m\}$ such that $k_{i_1}=k_{i_2}=k_{i_3}=k_{i_4}=k_{i_5}$.
\end{cor}

\section{Authomorphisms of Danielewski varieties}\label{las}

Let $A$ be a commutative associative algebra over $\KK$.
We say, that there is a {\it $\ZZ$-filtration} on $A$, if there are subspaces $A_i\subset A$, $i\in\ZZ$ such that if $i\geq j$, then $A_j\subset A_i$, $\cup_{i\in\ZZ}A_i=A$, 
$\cap_{i\in\ZZ}A_i=\{0\}$, and $A_i\cdot A_j\subset A_{i+j}$.  
\begin{de}
Let $A$ be an algebra with a fixed $\ZZ$-filtration on it. We can consider the {\it associated graded algebra} 
$$\mathrm{Gr}(A)=\bigoplus_{i\in\ZZ}A_i/A_{i-1}=\bigoplus_{i\in\ZZ}\overline{A}_i,$$
where if $\overline{f}=f+A_{i-1}\in\overline{A}_i$ and $\overline{g}=g+A_{j-1}\in\overline{A}_j$, then $\overline{f}\overline{g}=fg+A_{i+j-1}\in\overline{A}_{i+j}$.

There is a natural non-linear mapping $\mathrm{gr}\colon A\rightarrow \mathrm{Gr}(A)$, $f\mapsto \overline{f}$, where if $f\in A_i$ and $f\notin A_{i-1}$, then $\overline{f}=f+A_{i-1}\in\overline{A}_i$. By definition, $\mathrm{gr}(0)=0$.
\end{de}

Let $\partial$ be a nontrivial derivation of $A$. Fix $k\in\ZZ$. Suppose for all $i$ it is true, that $\partial(A_i)\subset A_{i+k}$.  Then we can define $\overline{\partial}_k\colon \overline{A}\rightarrow\overline{A}$ by the following rule.
If $f\in A_i\setminus A_{i-1}$, then
$$
\overline{\partial}_k(\overline{f})=\partial(f)+A_{i+k-1}\in\overline{A}_{i+k}, 
$$
$\overline{\partial}_k(0)=0$. On nonhomogeneous elements we define $\overline{\partial}_k$ by linearity. 
We obtain a homogeneous derivation of $\mathrm{Gr}(A)$ of degree $k$. If such $k$ exists and we take the minimal possible $k$, then $\overline{\partial}_k$ is a nontrivial derivation of~$A$. If $\partial$ is LND, then $\overline{\partial}_k$ is.

During this section by $X$ we denote a {\it Danielewski variety} 
$$
X=\VV(xy_1^{k_1}\ldots y_m^{k_m}-P(y_1,\ldots,y_m,z))\subset \KK^{m+2},
$$
where $m\geq 1$, for all $1\leq i\leq m$ we have $k_i\geq 2$, and 
$$
P(y_1,\ldots,y_m,z)=z^d+s_{d-1}(y_1,\ldots,y_m)z^{d-1}+\ldots+s_{0}(y_1,\ldots,y_m), d\geq 2,
$$
for some polynomials $s_j(y_1,\ldots,y_m)$.

There is a nontrivial LND of $\KK[X]$.
\begin{de}\label{cade} Let us define the {\it canonical LND} $\widetilde{\partial}$ on $X$ by 
$$\forall i\ \widetilde{\partial}(y_i)=0,\ \widetilde{\partial}(z)=y_1^{k_1}\ldots y_m^{k_m},\ \widetilde{\partial}(x)=\frac{\partial P}{\partial z}.$$ 
\end{de}
\begin{re}
If all $s_i$ are constants, then Danielewski variety $X$ is an $m$-suspension over a line considered in Section~\ref{oo} and the canonical derivation $\widetilde{\partial}$ coincides with the canonical derivation $\widehat{\partial}$ defined in Definition~\ref{cd}. 
\end{re}
This LND induces a degree function $\widetilde{\deg}$ on $\KK[X]\setminus\{0\}$ by the following rule. If $\widetilde{\partial}^p(f)\neq 0$ and $\widetilde{\partial}^{p+1}(f)=0$, then  $\widetilde{\deg}(f)=p$. This degree function gives a filtration: 
$$\KK[X]_i=\{f\in\KK[X]\mid \widetilde{\deg}(f)\leq i\}\cup \{0\}.$$

It is easy to see, that 
$$
\mathrm{Gr}(\KK[X])=\KK[\overline{x},\overline{y_1},\ldots, \overline{y_m}, \overline{z}]/(\overline{x}\overline{y_1}^{k_1}\ldots\overline{y_m}^{k_m}-\overline{z}^d).
$$
Let us denote the variety $\{\overline{x}\overline{y_1}^{k_1}\ldots\overline{y_m}^{k_m}-\overline{z}^d\}\subset\KK^{m+2}$ by $Y$.

\begin{re}
By Lemma~\ref{ir} the variety $Y$ is irreducible. That is $\KK[Y]$ does not admit any zero divisors. Therefore, $\KK[X]$ does not admit any zero divisors, that is $X$ is irreducible.
\end{re}

\begin{lem}\label{filtr}
Every automorphism $\psi$ of $X$ preserves the filtration of $\KK[X]$. That is $\psi^*(\KK[X]_i)=\KK[X]_i$.
\end{lem}
\begin{proof}

We can consider an embedding $\phi\colon\KK[X]\hookrightarrow \KK(y_1,\ldots,y_m, z)$ into the field of rational functions in $m+1$ variables given by 
$$\phi(x)=\frac{P(y_1,\ldots,y_m,z)}{y_1^{k_1}\ldots y_m^{k_m}}.$$

We need the following lemma due to Makar-Limanov.

\begin{lem}\cite[Lemma 2]{ML}
Let $\partial$ be a derivation of $A$, where $A$ is a subring of $\KK(x_1,\ldots,x_n)$. If the transcendence degree of $\Ker \partial$ is $n-1$ and $\{f_1,\ldots, f_{n-1}\}$ is a transcendence basis of $\Ker\partial$ then there exists a $g\in\KK(x_1,\ldots,x_n)$ such that 
$$\partial(a) = g\mathrm{J}(f_1,\ldots, f_{n-1}, a)$$ 
for every $a\in A$. 
Here $\mathrm{J}(f_1,\ldots, f_{n-1}, a)$ is the Jacobian relative to $x_1,\ldots,x_n$.
\end{lem}

It is easy to see, that if $\partial$ is a LND of $\KK[X]$, then all conditions of this lemma are satisfied. Therefore, 
$$
\partial(f)=g\mathrm{J}(f_1,\ldots,f_m,f).
$$

Let us put 
$$k=\widetilde{\deg}g+\sum_{i=1}^{m}\widetilde{\deg}f_i-\sum_{i=1}^{m}\widetilde{\deg}y_i-\widetilde{\deg}z=\widetilde{\deg}g+\sum_{i=1}^{m}\widetilde{\deg}f_i-1.$$
Then $\partial(A_i)\subset A_{i+k}$.

Let $k_0$ be the minimal $k$ satisfying  the condition $\partial(A_i)\subset A_{i+k}$. We obtain a nontrivial homogeneous LND $\overline{\partial}_{k_0}$ of $\KK[Y]$ of degree $k_0$. 
By Lemma~\ref{mlsul}(2) we obtain $\overline{\partial}_{k_0}=h(\overline{y_1},\ldots, \overline{y_m})\widehat{\partial}$, where for all $i$ we have $\widehat{\partial}(\overline{y_i})=0$, $\widehat{\partial}(\overline{z})=\overline{y_1}^{k_1}\ldots \overline{y_m}^{k_m}$, 
$\widehat{\partial}(\overline{x})=d\overline{z}^{d-1}$. It is easy to see, that $h(\overline{y_1},\ldots, \overline{y_m})\widehat{\partial}$ is a homogeneous LND of degree $-1$. We obtain 
$k_0=-1$. That is $\partial(\KK[X]_i)\subset \KK[X]_{i-1}$. Therefore, for every $f\in \KK[X]_i$ and for every LND $\partial$ of $\KK[X]$ we have $\partial^{i+1}(f)=0$.

Suppose there exist $\psi\in\Aut(X)$ and $f\in\KK[X]_i$ such that $\psi^*(f)\notin\KK[X]_i$. Then $\widetilde{\partial}^{i+1}(\psi^*(f))\neq 0$. Hence, 
$$
(\psi^*)^{-1}(\widetilde{\partial}^{i+1}(\psi^*(f)))=\left((\psi^*)^{-1}\circ \widetilde{\partial}\circ \psi^*\right)^{i+1}(f)\neq 0.
$$ 
But $(\psi^*)^{-1}\circ \widetilde{\partial}\circ \psi^*$ is an LND of $\KK[X]$. We obtain a contradiction. Therefore, for every automorphism $\psi$ we have
$\psi^*(\KK[X]_i)\subset\KK[X]_i$. Hence, $\psi^*(\KK[X]_i)=\KK[X]_i$. Lemma~\ref{filtr} is proven.
\end{proof}

\begin{prop}\cite[Theorem~2.8]{D}\label{dub}
Let $X$ be a Danielewski variety. Then the Makar-Limanov invariant of $X$ is
$\KK[y_1,\ldots, y_m]$.
\end{prop}
\begin{proof}
We have proved, that for every $f\in \KK[X]_i$ and for every LND~$\partial$ of $\KK[X]$ we have $\partial^{i+1}(f)=0$. Therefore, $\partial(y_j)=0$. Since $\Ker\widetilde{\partial}=\KK[y_1,\ldots, y_m]$, we obtain $\mathrm{ML}(X)~=~\KK[y_1,\ldots, y_m]$.
\end{proof}

\begin{lem}\label{opre}
Every $\psi\in\Aut(\KK[X])$ induces $\overline{\psi}\in\Aut(\KK[Y])$ by the rule $\overline{\psi}(\overline{f})=\overline{\psi(f)}$.
\end{lem}
\begin{proof}
First of all we need to check correctness of the definition of $\overline{\psi}$. If $\overline{f}=\overline{g}$, then there is $i$ such that $f\notin \KK[X]_i$ and $f-g\in\KK[X]_i$. Then $\psi(f)\notin\KK[X]_i$ and $\psi(f-g)\in\KK[X]_i$. Therefore, $\overline{\psi(f)}=\overline{\psi(g)}$.

We have 
$$
\overline{\psi}(\overline{f}\cdot\overline{g})=\overline{\psi}(\overline{f\cdot g})=\overline{\psi(f\cdot g)}=\overline{\psi(f)\cdot \psi(g)}=\overline{\psi(f)}\cdot\overline{\psi(g)}=\overline{\psi}(\overline{f})\cdot\overline{\psi}(\overline{g}).
$$
Hence, $\overline{\psi}$ is a homomorphism $\KK[Y]\rightarrow\KK[Y]$.

It is easy to see that $\overline{\psi}\circ\overline{\psi^{-1}}=\mathrm{id}$.
\end{proof}

\begin{lem}\label{homom}
The mapping $\Phi\colon\Aut(\KK[X])\rightarrow\Aut(\KK[Y])$, $\Phi(\psi)=\overline{\psi}$ is a homomorphism.
\end{lem}
\begin{proof}
For every $f\in\KK[X]$ we have:
\begin{multline*}
\Phi(\psi\circ\tau)(\overline{f})=\overline{(\psi\circ\tau)(f)}=\overline{\psi(\tau(f))}=\Phi(\psi)(\overline{\tau(f)})=\\
=\Phi(\psi)(\Phi(\tau)(\overline{f}))=(\Phi(\psi)\circ\Phi(\tau))(\overline{f}).
\end{multline*}

Therefore, $\Phi(\psi\circ\tau)=\Phi(\psi)\circ\Phi(\tau)$.
\end{proof}

Let us consider the automorphism $\alpha$ of $\KK^{m+2}$ given by 
$$\forall i\ \alpha^*(y_i)=y_i, \qquad \alpha^*(x)=x, \qquad\alpha^*(z)=z-\frac{s_{d-1}(y_1,\ldots,y_m)}{d}.$$ 
If we change coordinates via $\alpha$, $X$ maps to a variety with the same type of equation, but with $s_{d-1}(y_1,\ldots,y_m)=0$. 
Thus, further we assume that $s_{d-1}(y_1,\ldots,y_m)=0$.

Consider $\mathrm{S}_m$-action on $\KK[y_1,\ldots,y_m]$ by permutations of coordinates. By $S(X)\subset \mathrm{S}_m$ we denote the stabilizer of the monomial $y_1^{k_1}\ldots y_m^{k_m}$. That is $S(X)$ permutes only variables with coinciding~$k_i$.

Consider the action of the semidirect product $G=S(X)\rightthreetimes (\KK^\times)^{m+1}$ on $\KK(y_1,\ldots,y_m,z)$ given by
$$
(\sigma, t_1,\ldots,t_{m+1})\cdot (y_1,\ldots,y_m,z)=(t_1y_{\sigma(1)},\ldots,t_my_{\sigma(m)},t_{m+1}z).
$$
Recall that we have an embedding $\phi\colon\KK[X]\hookrightarrow \KK(y_1,\ldots,y_m,z)$. Let $\mathbb{G}\subset G$ be the stabilizer of $\phi(\KK[X])$. Then we have a $\mathbb{G}$-action on $X$. 
\begin{de}
Let us call $\mathbb{G}$ the {\it canonical group} of $X$.
\end{de}

If we know the equation of $X$, then we can explicitly describe $\mathbb{G}$. Indeed, if $g~=~(\sigma,t_1,\ldots, t_m)\in \mathbb{G}$ then 
$$
g\cdot x=\frac{P(t_1y_{\sigma(1)},\ldots,t_my_{\sigma(m)}, t_{m+1}z)}{\prod_{i=1}^m (t_i y_i)^{k_i}}.
$$
We have 
\begin{multline*}
g\cdot x-\frac{t_{m+1}^d}{\prod_{i=1}^m t_i^{k_i}}x=\frac{P(t_1y_{\sigma(1)},\ldots,t_my_{\sigma(m)}, t_{m+1}z)}{\prod_{i=1}^m (t_i y_i)^{k_i}}-
\frac{t_{m+1}^dP(y_1,\ldots,y_m, z)}{\prod_{i=1}^m (t_i y_i)^{k_i}}=\\
=\frac{\sum_{i=0}^{d-2}(t_{m+1}^is_i(t_1y_{\sigma(1)},\ldots,t_my_{\sigma(m)})-t_{m+1}^ds_i(y_1,\ldots,y_m))z^i}{\prod_{i=1}^m (t_i y_i)^{k_i}}.
\end{multline*}
Thus we can define $g\cdot x$ correctly if and only if 
$$
\prod_{i=1}^m ( y_i)^{k_i}\mid(s_i(t_1y_{\sigma(1)},\ldots,t_my_{\sigma(m)})-t_{m+1}^{d-i}s_i(y_1,\ldots,y_m))
$$
for all $0\leq i\leq d-2$. That is $\mathbb{G}$ is given in $G$ by conditions, that coefficients of all monomials in $s_i(t_1y_{\sigma(1)},\ldots,t_my_{\sigma(m)})-t_{m+1}^{d-i}s_i(y_1,\ldots,y_m)$, which a not divisible by 
$\prod_{i=1}^m ( y_i)^{k_i}$ are equal to zero.

\begin{ex}\label{e2}
Let $X=\VV(xy^2-(z^3+(y+1)z+1))$. Then $G\cong (\KK^\times)^2$. The nontrivial~$s_i$ are $s_1=y+1$ and $s_0=1$. The conditions to $\mathbb{G}$ are 
$$
\begin{cases}
y^2\mid (t_1y+1-t_2^2(y+1));\\
y^2\mid (1-t_2^3).
\end{cases}
$$
Hence, $t_1=t_2=1$, that is $\mathbb{G}=\{\mathrm{id}\}$.
\end{ex}

\begin{ex}
Let $X=\VV(xy_1^2y_2^2y_3^3-(z^2+y_1^2y_2^3y_3^4+y_3^3+1))$.
Then $G=\mathrm{S}_2\rightthreetimes(\KK^\times)^4$. The only nontrivial $s_i$ is $s_0(y_1, y_2, y_3)=y_1^2y_2^3y_3^4+y_3^3+1$. The condition to $\mathbb{G}$ is 
$$
y_1^2y_2^2y_3^3\mid (t_1^2t_2^3t_3^4y_{\sigma(1)}^2y_{\sigma(2)}^3y_3^4+t_3^3y_3^3+1-t_4^2(y_1^2y_2^3y_3^4+y_3^3+1)).
$$ 
Therefore $t_4^2=1$, $t_3^3=1$, that is $\mathbb{G}\cong \mathrm{S}_2\rightthreetimes \left((\KK^\times)^2\times\ZZ_3\times \ZZ_2\right)$.
\end{ex}

If $X$ is a suspension over a line considered in Section~\ref{oo}, then $\GG=S(X)\rightthreetimes (\TT\times\DD)$. But in general case $\GG$ can not be decomposed in such a semidirect product. Let us give an example, where $\mathbb{G}$ can not be decomposed into a semidirect product of $\mathbb{G}_1\subset \mathrm{S}_m$ and $\mathbb{G}_2\subset (\KK^\times)^{m+1}$.

\begin{ex}\label{e4}
Let $X=\VV(xy_1^2y_2^2-(z^3+z+y_1-y_2))$. Here $G=\mathrm{S}_2\rightthreetimes (\KK^\times)^3$. The nontrivial $s_i$ are $s_1=1$ and $s_0(y_1, y_2)=y_1-y_2$. The condition to $\mathbb{G}$ is
$$
\begin{cases}
y_1^2y_2^2\mid (1-t_3^2);\\
y_1^2y_2^2\mid (t_1y_{\sigma(1)}-t_2y_{\sigma(2)}-t_3^3(y_1-y_2)).
\end{cases}
$$

We have $t_3^2=1$ and
if $\sigma=\mathrm{id}$, then $t_1=t_3^3$ and $t_2=t_3^3$,
if $\sigma=(1,2)$, then $t_2=-t_3^3$, $t_1=-t_3^3$.

So, 
$$
\mathbb{G}=\{(\mathrm{id},1,1,1), (\mathrm{id}, -1,-1,-1), ((1,2), -1,-1,1), ((1,2),1,1,-1)\}.
$$
It is easy to see, that $\mathbb{G}\cong\ZZ_2\oplus\ZZ_2$, but it can not be decomposed into a product of any subgroups of $S_2$ and $(\KK^\times)^3$.
\end{ex}

\begin{prop}\label{mai}
Let $X$ be a Danielewski variety.
The group $\Aut(X)$ is generated by 

1) the canonical group $\mathbb{G}$ of $X$; 

2) the big unipotent subgroup $\UU(\widetilde{\partial})$ corresponding to the canonical LND on $X$.

\end{prop}
\begin{proof}
Let $\psi\in\Aut(\KK[X])$. Consider $\overline{\psi}=\Phi(\psi)$. By Lemma~\ref{mlsul}, there is $\sigma\in\mathrm{S}_m$ such that $\overline{\psi}(\overline{y_i})=\lambda_i \overline{y_{\sigma(i)}}$, $\lambda_i\in\KK\setminus\{0\}$ and $k_i=k_{\sigma(i)}$. 

We have $\psi(y_i)=\lambda_i y_{\sigma(i)}+g$, where $g\in\KK[X]_{-1}$. But $\KK[X]_{-1}=\{0\}$. Therefore, $\psi(y_i)~=~\lambda_i y_{\sigma(i)}$.

By Lemma~\ref{filtr}, the automorphism $\psi$ preserves the filtration of $\KK[X]$. Hence, 
$$\psi(z)=a(y_1,\ldots, y_m)z+b(y_1,\ldots, y_m).$$ 
But since $\psi$ is an automorphism, $a\in\KK\setminus\{0\}$.

Let us substitute this to the equation of $X$:
\begin{multline*}
\psi(x)\prod_{i=1}^m\lambda_i^{k_i}\prod_{i=1}^m y_{\sigma(i)}^{k_i}=(a z+b(y_1,\ldots, y_m))^d+\\
+s_{d-2}\left(\lambda_1y_{\sigma(1)},\ldots, \lambda_m y_{\sigma(m)}\right)(az+b(y_1,\ldots, y_m))^{d-2}+\ldots+\\
+s_{0}\left(\lambda_1y_{\sigma(1)},\ldots, \lambda_m y_{\sigma(m)}\right).
\end{multline*}
Since coefficients of $z^{d-1}$ on the left side and on the right side are equal, we obtain:
$$
\prod_{i=1}^m y_{\sigma(i)}^{k_i} \mid b(y_1,\ldots, y_m).
$$
That is 
$$b(y_1,\ldots, y_m)=y_{\sigma(1)}^{k_1}\ldots y_{\sigma(m)}^{k_m}h(y_{\sigma(1)},\ldots,y_{\sigma(m)}).
$$
Denote 
$$
\xi=\exp\left(-\frac{h(y_1,\ldots,y_m)}{a}\widetilde{\partial}\right).
$$
Then $(\xi\circ\psi)(y_i)=\lambda_i y_{\sigma(i)}$, $(\xi\circ\psi)(z)=a z$.
Therefore, $\xi\circ\psi\in \mathbb{G}$.
\end{proof}

The group $\mathbb{G}$ is finite extension of a torus. This follows that 
$$\mathbb{G}\cap\UU(\widetilde{\partial})=\{\mathrm{id}\}.$$ 
Since $\mathrm{SAut}(X)$ is a normal subgroup in $\Aut(X)$, we obtain the following theorem.

\begin{theor} \label{main}
Let $X$ be a Danielewski variety. Let $\mathbb{G}$ be the canonical group of $X$ and~$\UU(\widetilde{\partial})$ be the big unipotent subgroup corresponding to the canonical LND $\widetilde{\partial}$  on $X$. Then the automorphism group of $X$ is isomorphic to a semidirect product
$\mathbb{G}\rightthreetimes \UU(\widetilde{\partial})$.
\end{theor}

\begin{cor}
Let $X$ be a Danielewski variety. The group $\Aut(X)$ is commutative if and only if the canonical group $\mathbb{G}$ of $X$ is trivial.
\end{cor}
\begin{proof}
Let $g=(\sigma, t_1,\ldots, t_{m+1})\in\mathbb{G}$. For any $h\in\KK[y_1,\ldots, y_m]$ we can consider the automorphism $\alpha_h=\exp (h(y_1,\ldots,y_m)\widetilde{\partial})\in\UU(\widetilde{\partial})$. The condition that $g$ and $\alpha_h$ commute implies 
$$g\cdot (h(y_1,\ldots,y_m)y_1^{k_1}\ldots y_m^{k_m})=t_{m+1}h(y_1,\ldots,y_m)y_1^{k_1}\ldots y_m^{k_m}.$$
Since the action of $g$ is an automorphism, for all $h(y_1,\ldots,y_m)$ we have $g\cdot h=t_{m+1}h$. This is possible if and only if $t_{m+1}=1$ and $g=\mathrm{id}$.
\end{proof}

For an arbitrary Danielewski variety, the condition from Corollary~\ref{sold} does not give a criterium for automorphism group to be solvable. It gives only sufficient condition, which is not necessary. 

\begin{cor}
Let $X$ be a Danielewski variety. If there are no five indices $i_1, i_2, i_3, i_4, i_5~\in~\{1,\ldots, m\}$ such that $k_{i_1}=k_{i_2}=k_{i_3}=k_{i_4}=k_{i_5}$, then the group $\Aut(X)$ is solvable.
\end{cor}
\begin{proof}
Since the group $\UU(\widetilde{\partial})$ is commutative, $\Aut(X)$ is solvable if and only if $\mathbb{G}$ is solvable. Let us consider the homomorphism 
$\varphi\colon \mathbb{G}\rightarrow S(X)$ given by 
$\varphi(\sigma,t_1,\ldots,t_{m+1})=\sigma$. 
The kernel of $\varphi$ is contained in $(\KK^\times)^{m+1}$, hence, it is commutative. Therefore, if $S(X)$ is solvable, then $\Aut(X)$ is. As in Corollaries~\ref{sb} and~\ref{sold}, $S(X)$ is not solvable if and only if there are five coinciding $k_i$.
\end{proof}

\begin{ex}
Let $X=\VV(xy^2-(z^3+z(y+1)+1))$. Then $\mathbb{G}=\{\mathrm{id}\}$, see Example~\ref{e2}. Therefore, $\Aut(X)=\UU(\widetilde{\partial})$, where $\widetilde{\partial}(y)=0$, $\widetilde{\partial}(z)=y$, $\widetilde{\partial}(x)=3z^2+y+1$. These automorphisms have the following form.

$$
\begin{cases}
x\mapsto x+(3z^2+y+1)h(y)+3zy(h(y))^2+y^2(h(y))^3;\\
y\mapsto y;\\
z\mapsto z+yh(y).
\end{cases}
$$
\end{ex}

\begin{ex}
Let $X=\VV(xy_1^2y_2^2-(z^3+z+y_1-y_2))$. By Example~\ref{e4}, 
$$
\mathbb{G}=\{(\mathrm{id},1,1,1), (\mathrm{id}, -1,-1,-1), ((1,2), -1,-1,1), ((1,2),1,1,-1)\}\cong\ZZ_2\oplus\ZZ_2.
$$
We have, $\Aut(X)\cong(\ZZ_2\oplus\ZZ_2)\rightthreetimes \UU(\widetilde{\partial})$, where 
$\widetilde{\partial}(y_1)=\widetilde{\partial}(y_2)=0$, $\widetilde{\partial}(z)=y_1^2y_2^2$, $\widetilde{\partial}(x)=3z^2+1$.
Automorphisms from $\UU(\widetilde{\partial})$ have the following form.
$$
\begin{cases}
x\mapsto x+(3z^2+1)h(y_1,y_2)+3zy_1^2y_2^2(h(y_1,y_2))^2+y_1^4y_2^4(h(y_1,y_2))^3;\\
y_1\mapsto y_1;\\
y_2\mapsto y_2;\\
z\mapsto z+y_1^2y_2^2h(y_1,y_2).
\end{cases}
$$
\end{ex}

\end{document}